\documentclass[a4paper, oneside, english,reqno]{amsart}

\usepackage[utf8]{inputenx}
\usepackage{amsthm, amsmath, amssymb, amsfonts, mathrsfs, mathtools, stmaryrd}
\usepackage{babel, textcomp, url, enumerate, latexsym, graphicx, varioref, hyperref, multirow, layout, siunitx, booktabs}
\usepackage{pdflscape}
\usepackage{verbatim}
\usepackage[top=30mm,bottom=40mm,inner=30mm,outer=35mm]{geometry} 
\usepackage{dsfont}
\usepackage{thmtools}
\usepackage{amsmath}

\RequirePackage{enumitem}
\setlist[itemize]{font = \upshape, before = \leavevmode}
\setlist[enumerate]{font = \upshape, before = \leavevmode}
\setlist[description]{before = \leavevmode}

\usepackage[usenames]{color}
\usepackage{colortbl}

\usepackage{tikz, babel, rotating, tikz-cd, pigpen}
\usetikzlibrary{matrix, arrows, decorations.pathmorphing}

\usepackage{tikz, babel, rotating, tikz-cd, pigpen}
\usetikzlibrary{matrix, arrows, decorations.pathmorphing}
\usepackage{cleveref}
\RequirePackage[color       = blue!20,
                bordercolor = black,
                textsize    = tiny,
                figwidth    = 0.99\linewidth]{todonotes}

\usepackage[all]{xy}
\usepackage{microtype}
\usepackage{csquotes}
\RequirePackage[backend    = biber,
                sortcites  = true,
                giveninits = true,
                doi        = false,
                isbn       = false,
                url        = false,
                maxnames   = 6,
                style      = numeric]{biblatex}
\numberwithin{equation}{section}

\theoremstyle{plain}
\newtheorem{theorem}{Theorem}[section]
\newtheorem{lemma}[theorem]{Lemma}

\newtheorem{proposition}[theorem]{Proposition}

\newtheorem{corollary}[theorem]{Corollary}

\theoremstyle{definition}
\newtheorem{definition}[theorem]{Definition}

\newtheorem{example}[theorem]{Example}

\theoremstyle{remark}
\newtheorem{remark}[theorem]{Remark}

\title{Cousin complexes in motivic homotopy theory}

\author{Andrei Druzhinin}
\address{Chebyshev Laboratory, St. Petersburg State University  \& 
St. Petersburg Department of Steklov Mathematical Institute of Russian Academy of Sciences, Russia}
\email{\href{mailto:andrei.druzh@gmail.com}{andrei.druzh@gmail.com}}

\author{H{\aa}kon Kolderup}
\address{Department of Mathematics, University of Oslo, Norway}
\email{\href{mailto:hakon.kolderup@hotmail.com}{hakon.kolderup@hotmail.com}}

\author{Paul Arne {\O}stv{\ae}r}
\address{Department of Mathematics Federigo Enriques, University of Milan, Italy \& 
Department of Mathematics, University of Oslo, Norway}
\email{\href{mailto:paul.oestvaer@unimi.it}{paul.oestvaer@unimi.it} \&
\href{mailto:paularn@math.uio.no}{paularne@math.uio.no}}

\date{}

\newcommand{\Et}{\mathrm{\acute{E}t}}

\newcommand{\nis}{\mathrm{Nis}}
\newcommand{\zar}{\mathrm{Zar}}
\newcommand{\szar}{\mathrm{sZar}}
\newcommand{\Nis}{\mathrm{Nis}}

\newcommand{\tf}{\mathrm{tf}}
\newcommand{\zf}{\mathrm{zf}}
\newcommand{\szf}{\mathrm{szf}}

\newcommand{\Sm}{\mathrm{Sm}}
\newcommand{\EssSm}{\mathrm{EssSm}}
\newcommand{\Sch}{\mathrm{Sch}}
\newcommand{\Aff}{\mathrm{Aff}}
\newcommand{\Ab}{\mathbf{Ab}}

\newcommand{\SH}{\mathbf{SH}}
\newcommand{\A}{\mathbb A}
\newcommand{\Gm}{\mathbb G_m}
\newcommand{\ZF}{\mathbb Z\mathrm{F}}
\newcommand{\Fr}{\mathrm{Fr}}
\newcommand{\cFr}{\mathrm{cFr}}
\newcommand{\calO}{\mathcal{O}}
\newcommand{\gp}{\mathrm{gp}}
\newcommand{\fr}{\mathrm{fr}}

\newcommand{\Corr}{\mathrm{Corr}}
\newcommand{\Pre}{\mathrm{Pre}}

\newcommand{\Shv}{\mathrm{Shv}_\bullet}

\DeclareMathOperator{\codim}{codim}
\DeclareMathOperator{\cofib}{cofib}
\DeclareMathOperator{\fib}{fib}

\newcommand{\Smat}{\Sm^\mathrm{cci}}
\newcommand{\SmAff}{\mathrm{SmAff}}

\newcommand{\catPre}{\mathrm{Pre}}
\newcommand{\cPre}{\catPre}
\newcommand{\cPrefr}{\catPre^\fr}

\newcommand{\catSpt}{\mathrm{Spt}}

\newcommand{\cSpt}{\catSpt}
\newcommand{\cSptfr}{\catSpt^\fr}
\newcommand{\cSpts}{\catSpt}
\newcommand{\cSptst}{\cSpt^{s,t}}

\newcommand{\cSptsfr}{\catSpt^{\fr}}
\newcommand{\cSptstfr}{\cSpt^{\fr,s,t}}
\newcommand{\cSptPfr}{\cSpt^{\fr,\PP^1}}
\newcommand{\cPretf}{\cPre_\tf}
\newcommand{\cSptstf}{\catSpt_\tf}
\newcommand{\cPrefrtf}{\cPre^\fr_\tf}
\newcommand{\cSptsfrtf}{\cSptsfr_\tf}
\newcommand{\cSptstfrtf}{\cSptstfr_\tf}
\newcommand{\OmegaSigma}{\Theta}

\newcommand{\cSptsfrAtau}{\cSpt^{\fr}_{\A^1,\tau}}
\newcommand{\cSptstfrAtau}{\cSpt^{\fr,s,t}_{\A^1,\tau}}
\newcommand{\cSptPfrAtau}{\cSpt^{\fr,\PP^1}_{\A^1,\tau}}

\newcommand{\cSptstAN}{\cSpt^{s,t}_{\A^1,\nis}}

\newcommand{\cSptstfrAN}{\cSpt^{\fr,s,t}_{\A^1,\nis}}

\newcommand{\PP}{\mathbb{P}}
\newcommand{\SHs}{\SH^{s}}

\newcommand{\SHfrs}{\SH^{\fr,s}}
\newcommand{\SHsA}{\SH_{\A^1}^{s}}
\newcommand{\SHsAtau}{\SH_{\A^1,\tau}^{s}}
\newcommand{\SHfrsAtau}{\SH_{\A^1,\tau}^{\fr,s}}
\newcommand{\SHstAtau}{\SH_{\A^1,\tau}^{s,t}}
\newcommand{\SHfrstAtau}{\SH_{\A^1,\tau}^{\fr,s,t}}
\newcommand{\SHPAtau}{\SH_{\A^1,\tau}^{\PP^1}}
\newcommand{\SHfrPAtau}{\SH_{\A^1,\tau}^{\fr,\PP^1}}
\newcommand{\SHsAtf}{\SH_{\A^1,\tf}^{s}}

\newcommand{\SHfrsAtf}{\SH_{\A^1,\tf}^{\fr,s}}
\newcommand{\SHfrAnis}{\SH_{\A^1,\nis}^{\fr}}
\newcommand{\SHfrsAnis}{\SH_{\A^1,\nis}^{\fr,s}}

\newcommand{\SHsAN}{\SH_{\A^1,\nis}^{s}}

\newcommand{\SHfrstAnis}{\SH_{\A^1,\nis}^{\fr,s,t}}
\newcommand{\SHPAN}{\SH_{\A^1,\nis}^{\PP^1}}

\newcommand{\Map}{\mathrm{Map}}

\newcommand{\Cou}{C^\bullet}
\newcommand{\ovCou}{\mathrm{Cou}^\bullet}
\newcommand{\Coubi}{C^{\bullet,\bullet}}
\newcommand{\Kom}{\mathrm{Ch}}
\newcommand{\biKom}{\Kom^{\mathrm{bi}}}
\newcommand{\TCat}{\mathcal{S}}
\newcommand{\hTCat}{\mathrm{h}\TCat}
\newcommand{\hF}{\calF} 
\newcommand{\KC}{\Kom(\hTCat)}
\newcommand{\biKC}{\biKom(\hTCat)}
\newcommand{\Tot}{\mathrm{Tot}}
\newcommand{\SHtop}{\SH^{\mathrm{top}}}

\newcommand{\calF}{\mathcal{F}}
\newcommand{\calsimeq}{\simeq}

\newcommand{\calL}{\mathcal{L}}

\newcommand{\ovS}{\overline{S}}
\newcommand{\ctfr}{\mathrm{fr}}

\newcommand{\id}{\mathrm{id}}

\newcommand{\ii}[1]{{ii^!_!}}
\newcommand{\jj}[1]{{jj^*_*}}
\newcommand{\iis}[1]{{ii^*_*}}
\newcommand{\jjs}[1]{{jj^{\#}_{\#}}}

\newcommand{\bbZ}{\mathbb Z}

\newcommand{\SHANis}{\SH_{\A^1,\Nis}}

\newcommand{\ovX}{\overline{X}}

\newcommand{\ovcalX}{\overline{\calX}}
\newcommand{\calX}{\mathcal X}
\newcommand{\calC}{\mathcal{C}}
\newcommand{\ovcalC}{\overline{\calC}}
\newcommand{\ovcalZ}{\overline{\calZ}}
\newcommand{\calZ}{\mathcal{Z}}

\newcommand{\SigmainftyT}{\Sigma^\infty_{\PP^1}}

\begin{document}

\begin{abstract}
We investigate Cousin (bi-)complexes in the setting of motives. 
For essentially smooth local schemes, 
the columns of the Cousin bicomplex with coefficients in any stable 
motivic homotopy type is shown to be acyclic.
On the other hand, 
we also construct a family of non-acyclic Cousin complexes over any positive dimensional base scheme.
Our method of proof employs the notion of extended compactified framed correspondences.

Three major motivations for this study are to further our understanding of strict homotopy invariance, 
motivic infinite loop spaces, and connectivity in stable motivic homotopy theory.
As applications of our main results on motivic Cousin complexes, 
we generalize several fundamental results in these topics to finite dimensional base schemes.
\end{abstract}

\maketitle
\tableofcontents
\addtocontents{toc}{\protect\setcounter{tocdepth}{1}}

\section{Introduction}

\subsection{Cousin and Gersten complexes}

Cousin complexes of sheaves are ubiquitous in algebraic geometry \cite{zbMATH03336606}, 
complex analysis \cite{zbMATH03738047}, 
and representation theory \cite{zbMATH03609937}.
For example, 
filtering a smooth variety by codimension of support yields Cousin complexes for Bloch-Ogus theories;
see \cite{zbMATH03480765}.
More generally, 
by the same construction as in \Cref{sect:Cousincomplex}, 
to any Panin-Smirnov cohomology theory $A$ \cite[Definition 2.1]{PanafterSmirnovOrCoh} 
and a smooth scheme $X$ over some base scheme $B$, 
one can associate the Cousin complex  
(see \cite[Equation 19, Theorem 9.1]{PanafterSmirnovOrCoh})
\begin{equation}\label{eq:ovCouXA}
A(X) \to
\bigoplus\limits_{x\in X^{(0)}}
A_{x}(X)
\to \dots \to
\bigoplus\limits_{x\in X^{(i)}}
A_{x}(X)
\to \dots \to 
\bigoplus\limits_{x\in X^{({\dim X})}}
A_{x}(X).
\end{equation}
As in \cite{PanafterSmirnovOrCoh}
we write $A_{Z}(X)=\bigoplus_n A^{n}_{Z}(X)$ for any closed immersion $Z\not\hookrightarrow X$, 
and $A(X)=A_{\emptyset}(X)$. 
Moreover, 
$A_{x}(X)$ is the cohomology of the local scheme $X_{(x)}$ of $X$ at $x$ 
with support in $x\in X$, 
see \cite[Definition 2.1]{PanafterSmirnovOrCoh}.
These complexes encode important data in the cohomology theory of algebraic varieties, 
since they form the first page of the coniveau spectral sequence converging to $A(X)$ 
\cite[\S 1]{zbMATH01066289}, \cite[Theorem 9.1, Corollary 9.2]{Pan19}.
This spectral sequence essentially carries the information of all Gysin sequences in one package.
We refer
to \cite{zbMATH01066289} for an axiomatic setup that provides examples such as {\'e}tale cohomology,
de Rham cohomology, motivic cohomology, cycle modules, Hodge cohomology, algebraic $K$-theory,
logarithmic Hodge-Witt and de Rham-Witt cohomology. 
An analog of Cousin complexes in the context of stable homotopy theory is known as the 
chromatic resolution \cite{zbMATH03984103}.

When $A$ is defined on regular $B$-schemes and admits 
Gysin isomorphisms of the form $A_Z(Y)\simeq A(Z)$ for all reduced $0$-dimensional closed subschemes 
$Z$ of local regular $B$-schemes $Y$
(for instance, satisfied by derived Witt groups, by \cite[\S lll.1, Theorem 84]{MR2181829}) 
the Cousin complex \eqref{eq:ovCouXA} specializes in the Gersten complex whose 
terms are given by direct sums of copies of $A(x):=A(k(x))$, $x\in X$. 

In particular, 
in degree $n\in\bbZ$, 
the Gersten complex for algebraic $K$-theory takes the familiar form
\begin{equation}
\label{eq:Ger}
K_{n}(X) 
\to
\bigoplus\limits_{x\in X^{(0)}}
K_{n}(x)
\to \bigoplus\limits_{x\in X^{(1)}}
K_{n-1}(x)
\to \dots \to
\bigoplus\limits_{x\in X^{({\dim X})}}
K_{0}(x).
\end{equation}
The Gersten conjecture says that \eqref{eq:Ger} is acyclic when $X$ is a 
local regular scheme \cite{Ge}.
Quillen's fundamental article \cite[\S 7]{Quillen} proves Gersten's conjecture 
for local schemes of the form $X_{(x)}$; here $X$ is a smooth scheme defined over a field and $x\in X$.
In \cite{Gillet-Levine}, 
Gillet and Levine proved a relative version of Gersten's conjecture, 
which reduces the conjecture for local rings smooth over a discrete valuation ring to the 
case of discrete valuation rings.
In \cite{Panin-equichar},
Panin proved that Gersten's conjecture holds for any equi-characteristic regular local ring.
More recently, 
Schmidt and Strunk \cite{Schmidt-Strunk-Bloch-Ogus} have shown a conditional exactness result 
for the Gersten complex defined by an $\A^1$-invariant Nisnevich local cohomology theory over 
Dedekind schemes with infinite residue fields. 
Deshmukh, Kulkarni and Yadav \cite{Nis-local-Bloch-Ogus} proved the exactness of the Gersten 
complex of essentially smooth local Henselian schemes over certain irreducible noetherian
base schemes, 
and obtained a version of the Bloch-Ogus theorem \cite{zbMATH03480765} 
in the Nisnevich topology.

A fundamental part of the Gersten conjecture is the injectivity of the left-most morphism in \eqref{eq:Ger}.
The Grothendieck-Serre conjecture \cite{Gr}, \cite{Se} is an analogous 
statement for the \'etale cohomology group $H^1_\mathrm{et}(-,G)$ 
-- defied by \'etale torsors for a reductive algebraic group $G$.
One may ask whether \eqref{eq:Ger} is acyclic for other $X$;
e.g., when $X=Y-Z$, where $Y$ is a regular local scheme, 
and $Z$ is a regular divisor on $Y$. 
An analogous situation for \'etale torsors is the subject of
Nisnevich's conjecture \cite[Conjecture 1.3]{zbMATH04125570} 
studied in \cite{FedorovprovedNisnevichconjecture}. 
When $Y$ is an essentially smooth local scheme over a DVR with closed fiber $Z$, 
Bloch proved the acyclicity of \eqref{eq:Ger} for $X=Y-Z$ in \cite{MR0862630}. 
When $X=Y$ it follows that \eqref{eq:Ger}
is quasi-isomorphic to the complex 
\begin{equation}
\label{eq:tfGerstKdvr}
    K_n(X)\to K_{n}(X-Z)\to K_{n-1}(Z).
\end{equation}
Consequently, 
under these assumptions, 
the Gersten complex \eqref{eq:Ger} is exact  
beyond the term $\bigoplus_{x\in X^{(1)}}K_{n-1}(x)$. 
For reasons explained in \Cref{sect:tfCousincomplexdefinition}, 
\eqref{eq:tfGerstKdvr} is an example of a $\tf$-Cousin complex
(a notion we will study for a class of cohomology theories). 

\subsection{Motivic Cousin complexes}
\label{subsection:MotivicCousincomplexes}
Motivic homotopy theory, 
initiated by Morel and Voevodsky \cite{Morel-Voevodsky}, 
is a homotopy theory of schemes based on the idea that the affine line plays the role of the unit interval. 
Its stable version obtained by inverting the projective line is essential in the theory of motives and motivic cohomology.
Early successes include resolutions of the Bloch-Kato and Milnor conjectures relating $K$-theory, Galois cohomology, and quadratic forms. 

When the base scheme is a field $k$,
the Cousin complex \eqref{eq:ovCouXA} is acyclic for schemes of the form $X_{(x)}$ 
and cohomology theories represented in 
the $S^1$-stable motivic homotopy category $\SHsAN(k)$ or in the stable motivic homotopy category $\SHANis(k)$,  
see \cite{mot-functors}, \cite{Jardine-spt}, \cite{morel-trieste}, \cite{Voe98}.
A geometric presentation lemma strengthening Noether normalization was used in 
\cite{zbMATH03480765}, \cite{Quillen}, 
and refined in \cite{Morel-connectivity,Morel-over-a-field}, \cite{Pan19} 
to prove our claim for $\SHsAN(k)$.
Voevodsky's work on pretheories \cite{Voe-hty-inv} was originally applied to cohomology theories in the 
derived category of motives $\mathbf{DM}(k)$, 
and later to prove our claim for $\SHANis(k)$ via the theory of framed transfers as in 
\cite{Voe-notes,hty-inv,surj-etale-exc,nonperfect-SHI}.

This paper aims to study Cousin 
complexes over more general base schemes via stable motivic homotopy theory.
We fix a separated noetherian scheme $B$ of finite (positive) dimension $d$ .
As shown in \Cref{section:counterex}, 
a direct analog of the Gersten conjecture fails 
for cohomology theories representable 
in the stable motivic homotopy category $\SHANis(B)$.
If $X$ is an essentially smooth local $B$-scheme, 
we prove that the associated Cousin complex \eqref{eq:ovCouXA} is exact above degree $d$,
see \Cref{thm:trivcohhdimBCous},
while in general it has a non-trivial $d$-th cohomology group,
see \Cref{th:nontrivtermsBBz}
(in fact, 
the length $d+2$ complex in \eqref{eq:ovCoutfXA} is quasi-isomorphic 
to the Cousin complex for any motivic spectrum over $B$).

\subsubsection{Cousin complexes}
For any $\calF\in\SHANis(B)$ and closed embedding $Z\not\hookrightarrow X$ in $\Sm_{B}$, 
we set 
\begin{equation}
\label{eq:AYXpicalF}
A_Z^n(X):=[\Sigma^\infty_{\PP^1}(X/X-Z)_+,S^{n}\wedge\calF]_{\SHANis(B)},
\end{equation} 
and consider the associated complex $\pi\ovCou(X,\calF)$ as in \eqref{eq:ovCouXA}.
This is a complex of abelian groups of length $\dim X+2$.
It coincides with 
the complex obtained by taking homotopy groups of the terms in the complex $\ovCou(X,\calF)$
in the stable homotopy category $\SHtop$, 
see \Cref{def:ovCou}.

\begin{theorem}[\protect{\Cref{thm:trivcohhdimBCous,th:nontrivtermsBBz}}]
\label{th::CousComplex} 
Suppose $U$ is an essentially smooth local scheme over a noetherian separated scheme $B$ of finite Krull dimension.  
Moreover, let $z\in B$ denote the image of its closed point. 
Then, for any motivic spectrum $\calF\in \SH_{\A^1,\nis}(B)$, we have
\[
H^i(\pi\ovCou(U,\calF))=0,\quad i>\codim_B z.
\]
Moreover, 
when $\calF=\Sigma^\infty_{\PP^1}(B/(B-Z))$, where $Z$ is the closure of $z$ in $B$, 
we have
\[
H^{i}(\pi\ovCou( U, \calF ))\neq 0,\quad i=\codim_B z.
\]
\end{theorem}

To prove the above theorem, we introduce $\tf$-Cousin complexes.

\subsubsection{$\tf$-Cousin complex}
Let $\pi\ovCou_\tf(X,\calF)$ denote the complex
\begin{equation}
\label{eq:ovCoutfXA}
A(X) \to
\bigoplus\limits_{z\in B^{(0)}}A_{X_z}(X)
\to \dots \to
\bigoplus\limits_{z\in B^{({i})}}
A_{X_z}(X)
\to \dots \to 
\bigoplus\limits_{z\in B^{({\dim B})}}
A_{X_z}(X).
\end{equation}
Here, 
$X_z=X\times_B z$,
$A_{X_z}(X)=A_{X_z}(X\times_B B_{(z)})$,
and $B_{(z)}$ is the local scheme of $B$ at $z\in B$.
Note that \eqref{eq:ovCoutfXA} has length $\dim B+2$.
We refer to \eqref{eq:ovCoutfXA} as the $\tf$-Cousin complex of $\calF$ and $X$
because it is connected to the $\tf$-topology introduced in \cite[Definition 3.1]{DKO:SHISpecZ} 
the same way that the Cousin complex $\pi\ovCou(X,\calF)$ is connected to the Nisnevich topology; 
see \Cref{sect:Cousincomplex,sect:tfCousincomplexdefinition} for more details.

\begin{theorem}[\protect{\Cref{thm:trivcohhdimBCous}, \Cref{cor:CoutfsimeqCou}}]
\label{th::tfCousComplex}
For every $\calF\in \SH_{\A^1,\nis}(B)$, 
$U\in \EssSm_B$, 
there is a quasi-isomorphism of complexes of abelian groups
\begin{equation}\label{eq:pilCoutfpilCou}
    \pi\ovCou_\tf(U,\calF)\xrightarrow{\simeq} \pi\ovCou(U,\calF).
\end{equation}
\end{theorem}

\subsubsection{Cousin bicomplexes}
If $z\in B$, 
we form the pullback $U_z=X\times_B z$. 
We write $j\colon B_{(z)}\to B$ for the canonical morphism from the local scheme of $B$ at $z$, 
and $i\colon z\not\hookrightarrow B_{(z)}$ denotes the corresponding closed immersion.
The corresponding base change functors $i^!$ and $j^*$ induce the Cousin bicomplex 
comprised of $\dim B+1$ columns
\begin{equation}
\label{eq:bicomplex}
\bigoplus_{z\in B^{(0)}} \pi\ovCou(U_z,i^! j^*\calF) \to \dots \to  
\bigoplus_{z\in B^{(\dim B)}} \pi\ovCou(U_z,i^! j^*\calF).    
\end{equation}

We show that the totalization of \eqref{eq:bicomplex} is quasi-isomorphic to the cone of 
\eqref{eq:pilCoutfpilCou}.
The quasi-isomorphism \eqref{eq:pilCoutfpilCou} follows from the acyclicity of the 
columns in \eqref{eq:bicomplex}.
We generalize the results for DVRs in \cite{nonperfect-SHI} in the following way.

\begin{theorem}[\protect{\Cref{th:FnisExactoffiber}}]
\label{th::CousbiComplexcolumn}
For every $z\in B$, 
$\calF\in\SHANis(z)$,
and essentially smooth local scheme $U\in\EssSm_B$, 
the complex $\pi\ovCou(U_z,\calF)$ is acyclic.
\end{theorem}

\subsection{Framed motives, infinite loop spaces, and connectivity}
We show that the results in \Cref{subsection:MotivicCousincomplexes} hold more generally in the category 
$\SHfrsAtf(B)$ of framed $S^1$-spectra of $\A^1$-invariant $\tf$-local presheaves on $\Sm_{B}$.
As a consequence, 
we generalize the results on strict homotopy invariance and motivic infinite loop spaces 
in \cite{DKO:SHISpecZ}, see below.
Moreover, 
in combination with
\cite{SHfrzf}, 
 we improve the connectivity theorems for motivic spectra in
\cite{ConnDodekindDomains,ConnGabPresLemNoethDominffield,ConnBase} to essentially smooth local schemes.

\subsubsection{Homotopy invariance and infinite loop spaces} 
We prove the following improvement of the strict homotopy invariance theorem for 
framed presheaves in \cite[Theorem 12.2]{DKO:SHISpecZ}, see \Cref{subsection:notation} for notation.
Recall that $B$ is any base scheme.

\begin{theorem}[$\tf/\nis$-strict homotopy invariance, \Cref{th:SHIsepnfdB}]
\label{th:strhominv::SHIsepnfdB}
For any $\calF\in \SH^{\fr,s,t}_{\A^1,\tf}(B)$,
the Nisnevich localized presheaf of $S^1$-spectra $L_\nis\calF$ is $\A^1$-invariant.
\end{theorem}

For every $X\in\Sm_B$ we associate the infinite motivic loop space
\[\Omega^\infty_{\PP^1}\Sigma^\infty_{\PP^1}X_{+}\]
for the loop-suspension adjoint functors $\Sigma^\infty_{\PP^1}$ and $\Omega^\infty_{\PP^1}$.
Our next result identifies this construction in terms of framed correspondences, group completions, $\Gm$ loops,
and localization functors for $\A^1$-invariance and Nisnevich and $\tf$-topologies.

\begin{theorem}[\Cref{th:LnisOmegaGmLA1tfFrSGmXgp}]
\label{th::LnisOmegaGmLA1tfFrSGmXgp}
For any $X\in \Sm_B$ there is a natural equivalence
\begin{equation}\label{eq:th::LNisOmegaGmLA1tfFrSGmXgp}
\Omega^\infty_{\PP^1}\Sigma^\infty_{\PP^1}X_{+}\simeq 
L_\nis \varinjlim_l \Omega^l_{\Gm}(L_{\A^1,\tf} \Fr(-, X_{+}\wedge\Gm^{\wedge l} ))^\gp.
\end{equation}
\end{theorem}

\subsubsection{Zariski localization}

Using the results on the framed motives in \cite{SHfrzf},
we improve \Cref{th::LnisOmegaGmLA1tfFrSGmXgp} by showing that the 
Zariski localization of the colimit in \eqref{eq:th::LNisOmegaGmLA1tfFrSGmXgp} agrees with 
the Nisnevich localization.

\begin{theorem}[\Cref{th:LzarOmegaGmLA1tfFrSGmXgp}]
\label{th::LzarOmegaGmLA1tfFrSGmXgp}
For any $X\in \Sm_B$, there is a natural equivalence
\begin{equation}\label{eq:th::LZarOmegaGmLA1tfFrSGmXgp}\Omega^\infty_{\PP^1}\Sigma^\infty_{\PP^1}X_{+}\simeq 
L_\zar \varinjlim_l \Omega^l_{\Gm}(L_{\A^1,\tf} \Fr(-, X_{+}\wedge\Gm^{\wedge l} ))^\gp.\end{equation}

Moreover,
for any semi-local scheme $U\in\EssSm_B$,
there is an equivalence
\begin{equation*}
(\Omega^\infty_{\PP^1}\Sigma^\infty_{\PP^1}X_{+})(U)\simeq
\varinjlim_l \Omega^l_{\Gm}(L_{\A^1,\tf} \Fr(U, X_{+}\wedge\Gm^{\wedge l} ))^\gp,\end{equation*}
where the left side is the value of $\Omega^\infty_{\PP^1}\Sigma^\infty_{\PP^1}X_{+}$ on $U$.
\end{theorem}

This has the following application towards connectivity of motivic spectra.

\subsubsection{Connectivity}
Stable motivic homotopy sheaves of $\calF\in\SHANis(B)$ on $\Sm_B$ 
are the Nisnevich sheaves associated with the presheaves of abelian groups
\begin{equation}\label{eq:piijANisCalFU}
U\mapsto
\pi_{i,j}^{\A^1,\nis}(\calF)(U) = 
\mathrm{Hom}_{\SH_{\A^1,\nis}(B)}(\Sigma^\infty_{\PP^1}(S^{i-j}
\wedge \Gm^{\wedge j}\wedge U_+), \calF).
\end{equation}
Morel's connectivity theorem for motivic spectra over any field $k$ shows the vanishing 
\begin{equation}
\label{eq:piizSPXtriv}
\pi_{i,j}^{\A^1,\nis}(\Sigma^\infty_{\PP^1}X_+)(U)=0,
\end{equation}
for all $i<j\in\bbZ$, $X\in\Sm_k$,
and any essentially smooth local henselian scheme $U$ over $k$.
Indeed, 
according to \cite{morel-trieste,Morel-connectivity}
the vanishing in \eqref{eq:piizSPXtriv} holds for any essentially smooth local $k$-scheme $U$.
In \cite{ConnDodekindDomains,ConnGabPresLemNoethDominffield,ConnBase}, 
the vanishing was generalized to finite dimensional noetherian separated base schemes $B$ 
in the range $i<-\dim B+j$, $j\in\bbZ$,
and any essentially smooth local henselian $B$-scheme $U$.
We generalize the latter result to all essentially smooth semi-local $B$-schemes.

\begin{theorem}[\Cref{cor:ShifteddimBConnectivityZar}]\label{cor::ShifteddimBConnectivityZar}
    Suppose that $X\in\Sm_B$, 
    $i<-\dim B+j$, and $j\in\bbZ$.
    Then, for any essentially smooth semi-local
 $B$-scheme $U$,
    we have
    \[\pi^{\A^1,\nis}_{i,j}(\Sigma^\infty_{\PP^1} X_+)(U)=0.\]
\end{theorem}

\subsubsection{Framed correspondences}
\label{sect:intro:comcorrtheorem}

Voevodsky's lemma shows that the category $\ZF_*(B)$ of linear framed correspondences over $B$ 
admits a functor $\ZF_*(B)\to\SHANis(B)$;
for background, 
we refer to 
\cite[Definition 2.3, Lemma 3.2, Corollary 3.6, Proposition 3.8]{Framed} and
\cite[Definition 2.11]{hty-inv}, \cite[Equation 25]{RigiditySmAffPairs}.
Combined with \Cref{cor::InjectivityFrgenfibsmoothloc}, 
our main technical result on linear framed correspondences, 
it follows that for any cohomology theory $A$ in $\SHANis(B)$ there is a naturally induced injective homomorphism
\begin{equation}
\label{equation:injective}
A(X_{x,\eta})\to A(\theta).
\end{equation}
Here $B$ is a local scheme, 
and $\theta$ is the generic point of $X_{x,\eta}$ as defined below for 
$X\in \Sm_B$, $x\in X$, and $\eta\in B$.

\begin{theorem}[\Cref{cor:InjectivityFrgenfibsmoothloc}]\label{cor::InjectivityFrgenfibsmoothloc}
Suppose that the base scheme $B$ is local. 
Let $\eta\in B$ be a point.
Let $X\in \Sm_B$, $X_\eta=X\times_B \eta$, 
and $X_{x,\eta}=X_x\times_X X_{\eta}$, where $x\in X$ is a point.
Let $Z$ be a closed subscheme of positive codimension in $X_\eta$.
Then there is a framed correspondence $c\in \ZF_n(X_{x,\eta}\times\A^1, X_{\eta})$ such that
\begin{itemize}
\item[(1)] $c\circ i_0=\sigma^n\mathrm{can}$, 
where $\mathrm{can}\colon X_{x,\eta}\to X_{\eta}$ is the canonical morphism,
and $\sigma^n\mathrm{can}$ is the $n$th $\sigma$-suspension of the corresponding 
element in $\ZF_0(X_{x,\eta}, 
X_{\eta})$
\item[(2)] $c\circ i_1=j\circ \tilde c$ for some $\tilde c\in\ZF_n(X_{x,\eta}, X_{\eta}-Z)$, 
where $j\colon X_{\eta}-Z\to X_{\eta}$ is the canonical open immersion
\end{itemize}
Here, 
$i_j\colon X_{x,\eta}\to X_{x,\eta}\times\A^1$ denotes the $j$-section, for $j=0,1$.
\end{theorem}

The key new ingredient in the proof of \Cref{cor::InjectivityFrgenfibsmoothloc}
is the notion of extended compactified framed correspondences over local base schemes developed in \Cref{sect:FocCompCor}. 
Our constructions consist informally of 
increasing the relative dimension of the support over the source of the framed correspondence, 
compactifying the support relatively over the source, 
and compactifying either the source of the framed correspondence or the base scheme, 
and taking into account the closed fiber of such a compactification.
These ideas allow us to carry out an induction argument generalizing results for fields and 
$1$-dimensional base schemes from \cite{hty-inv,nonperfect-SHI,DKO:SHISpecZ} to arbitrary base schemes.

\subsection{Notation}
\label{subsection:notation}
Throughout, 
$B$ is a finite dimensional separated noetherian base scheme, 
$\Aff_B$ denotes affine $B$-schemes,
$\tau$ is a topology on smooth finite type $B$-schemes $\Sm_{B}$, 
and $\EssSm_B$ denotes essentially smooth $B$-schemes
(filtered limits with affine \'etale transition morphisms in $\Sm_B$).
We write $\dim^x_B X$ for the relative dimension of $X$ over $B$ at $x\in X$, and $\dim_B X=\max_{x\in X}\dim^x_B X$.
For $X\in \Sch_B$ and $z\in B$, 
we set $X_z:=X\times_B z$.
We write $X_{(x)}$ and $X^h_x$ for the local and the local henselian schemes at $x\in X$, 
respectively. 
If $Z\not\hookrightarrow X$ is a closed immersion and $X$ is affine, 
we write $X^h_Z$ for the corresponding henselization.
Moreover, 
we use the following notation.
\begin{center}
\begin{tabular}{l|l}
$\Corr^\ctfr(B)$ & 
$\infty$-category of tangentially framed correspondences\\
$\Pre(B)$, $\Pre^\fr(B)$ & $\infty$-categories of presheaves on $\Sm_B$, 
$\Corr^\fr(B)$ \\
$\Pre_{\A^1}(B)$, $\Pre^\fr_{\A^1}(B)$
& subcategories of $\A^1$-invariant presheaves \\
$\Pre_{\tau}(B)$, $\Pre^\fr_{\tau}(B)$ & subcategories of $\tau$-local presheaves \\
$\Pre_{\tau,\A^1}(B)$, $\Pre^\fr_{\tau,\A^1}(B)$ & 
$\Pre_{\A^1}(B)\cap \Pre_{\tau}(B)$, $\Pre^\fr_{\A^1}(B)\cap \Pre^\fr_{\tau}(B)$ \\
$\catSpt$, $\SHtop$, $\SH^{s}(B)$ & the $\infty$-category of spectra, 
stable homotopy category, \\
& stable homotopy category of $S^1$-spectra on $\Sm_B$ \\
$\cSpt(B)$, $\cSpt^{s,t}(B)$, $\cSpt^{\PP^1}(B)$ & 
$\infty$-categories $\cPre(B)[T^{-1}]$ for $T=S^1,S^1\wedge\Gm,\PP^1$\\
$\cSptsfr(B)$, $\cSptstfr(B)$, $\cSptPfr(B)$ & 
  $\infty$-categories $\cPrefr_{\tau,\A^1}(B)[T^{-1}]$ for $T=S^1,S^1\wedge\Gm,\PP^1$\\
$\SHsAtau(B)$, $\SHstAtau(B)$, $\SHPAtau(B)$ & 
  $\tau$-motivic stable homotopy category of $S^1$-spectra,
  \\ & $(s,t)$-bispectra, $\PP^1$-spectra on $\Sm_B$\\
$\cSpt_{\A^1,\tau}(B)$, $\cSpt^{s,t}_{\A^1,\tau}(B)$, $\cSpt^{\PP^1}_{\A^1,\tau}(B)$ & 
  $\infty$-categories $\cPre_{\tau,\A^1}(B)[T^{-1}]$ for 
  $T=S^1,S^1\wedge\Gm,\PP^1$\\
$\SHfrsAtau(B)$, $\SHfrstAtau(B)$, $\SHfrPAtau(B)$ & 
  $\tau$-motivic stable homotopy category of framed
  \\& $S^1$-spectra, $(s,t)$-bispectra, $\PP^1$-spectra on $\Sm_B$\\
$\cSptsfrAtau(B)$, $\cSptstfrAtau(B)$, $\cSptPfrAtau(B)$ & 
  $\infty$-categories $\cPrefr_{\tau,\A^1}(B)[T^{-1}]$ for 
  $T=S^1,S^1\wedge\Gm,\PP^1$\\
$L_\tau$, $L_{\A^1}$ & $\tau$-localization, $\A^1$-localization \\
$\Kom(\mathscr C)$, $\biKom(\mathscr C)$, $\Kom_{\geq 0}(\mathscr C)$ & 
(bi)complexes in an additive category $\mathscr C$ \\
& complexes concentrated in non-negative degrees\\
$\Tot\colon\biKom(\mathscr C)\to\Kom(\mathscr C)$, $\mathrm{Cone}(f)$
& totalization, cone of a morphism $f$ in $\Kom(\mathscr C)$ \\ 
$\pi_l\colon \Kom(\SHtop)\to\Kom(\Ab)$ & term-wise application of homotopy 
$\pi_l\colon \SHtop\to\Ab$\\
$\pi^{\A^1,\nis}_{i,j}(X)$ & 
stable motivic homotopy presheaves of $X\in\Sm_B$ 
\end{tabular}
\end{center}

\subsection{Acknowledgements}
The authors acknowledge support from the RCN Project no.~312472 ``Equations in Motivic Homotopy."
Our work is supported by The European Commission -- Horizon-MSCA-PF-2022 ``Motivic integral $p$-adic cohomologies"
and a Young Russian Mathematics award.
The Russian Science Foundation grant 19-71-30002 provided generous and specific support for 
\Cref{th::LzarOmegaGmLA1tfFrSGmXgp,cor::ShifteddimBConnectivityZar,cor::InjectivityFrgenfibsmoothloc}, 
and \Cref{sect:FocCompCor,sect:Zarstalks} with respect to the grant conditions.

\section{Framed motivic spectra}
\label{sect:PreSptFr}

In this section, 
we recall the notion of framed motivic spectra and the strict homotopy invariance theorem for 
framed presheaves over a field, 
which will be used later in the paper.
\subsection{Voevodsky's framed correspondences}\label{subsect:VoevFr}
First we recall Voevodsky's notion of framed correspondences \cite{Voe-notes,Framed,five-authors}:

\begin{definition}[Framed correspondences]\label{def:VoevFr}
Let $X$, $Y$ be schemes, and $V\subset Y$ an open subscheme.
An \emph{explicit framed correspondence of level $n$ from $X$ to $Y/V$} is a span
\begin{equation}\label{framed-span}
\begin{tikzcd}
& S\ar{dl}[swap]{f}\ar{dr} &\\
X & & Y
\end{tikzcd}
\end{equation}
where $f$ is finite, together with the data of
\begin{itemize}
\item[(1)] a closed embedding $S\not\hookrightarrow \A^n_X$ 
\item[(2)] an étale neighborhood $U$ of $S$ in $\A^n_X$ together with an $n$-tuple of regular functions 
$\varphi=(\varphi_1,\dots,\varphi_n)$ on $U$, and
\item[(3)] a morphism $g\colon U\to Y$ lifting the morphism $S\to Y$,
such that $S=U\times_{\A^n\times Y} (\{0\}\times Y\setminus V)$
\end{itemize}
Refining étale neighborhoods yields a notion of equivalent explicit framed correspondences. 
The resulting equivalence classes $\Fr_n(X,Y/V)$ are called level $n$ framed correspondences. 
Finally, 
we set $\Fr_n(X,Y):=\Fr_n(X,Y/\emptyset)$.
\end{definition}

\begin{example}\label{ex:langleurangle}
    Given a scheme $X$ and an invertible regular function $\nu\in\mathcal O_X(X)$, 
     let $\langle \nu\rangle\in\Fr_1(X,X)$ denote the framed correspondence 
     $(\A^1\times X,0\times X,\mathrm{pr}_X^*\nu t,\mathrm{pr}_X)$,
     where $t$ is the coordinate function on $\A^1$, 
     and $\mathrm{pr}_X\colon\A^1\times X\to X$ is the canonical projection.
Moreover, we set $\sigma_X=\langle 1\rangle$. 
\end{example}

\begin{remark}
In the case when $X$ is affine, 
we can form the henselization $(\A^n_X)_S^h$ of $S$ in $\A^n_X$; 
this is the universal étale neighborhood of $S$ in $\A^n_X$. 
Therefore, for affine $X$, 
the framed correspondences from $X$ to $Y$ can be described more succinctly as spans as above, with 
$U=(\A^n_X)^h_S$, without passing to equivalence classes.
\end{remark}

\begin{lemma}[Voevodsky's Lemma]\label{lm:VoevFrLemmaShvbullet}
For any $X,Y\in \Sm_B$ and $l\in \mathbb Z, l\geq 0,$ there is a canonical equivalence
\[\Map_{\Shv(B)}(X_+\wedge \PP^{\wedge l},Y_+\wedge T^{\wedge l})\simeq \Fr_l(X,Y),\]
where $\Shv(B)$ denotes the category of Nisnevich sheaves of pointed sets on $\Sm_B$.
\end{lemma}

Garkusha and Panin's work \cite{Framed} introduces the graded category of framed correspondences $\Fr_*(B)$
enriched over pointed sets, with hom-sets given by $\bigvee_{n}\Fr_n(X,Y)$.
We refer to \cite{Framed} for the definition of the composition in $\Fr_*(B)$.
Moreover, 
every regular map $X\to Y$ defines an element in $\Fr_0(X,Y)$ for any isomorphism 
$f\colon S\xrightarrow{\cong} X$. 
There is a canonically induced functor
\begin{equation}\label{eq:SmBtoFrstarB}
    \Sm_B\to \Fr_*(B).
\end{equation}

\begin{definition}\label{def:sigmasuspFr}
For $X,Y\in \Sm_B$ the $n$th $\sigma$-suspension map  
\begin{equation}
\label{eq:sigmanc}
\sigma^n
\colon 
\Fr_l(X,Y) \to  \Fr_{l+n}(X,Y)
\end{equation}
is given by 
$$
c=(U,S,(\phi),g) \mapsto \sigma^n c=(U,S,(\phi,t_{l+1},\dots ,t_{l+n}),g).
$$
Moreover, we set $\sigma^n_Y=\sigma^n\id_Y$.
\end{definition}
\begin{remark}
    \Cref{def:sigmasuspFr} specializes to \Cref{ex:langleurangle}
    in the sense that $\sigma_X=\sigma^1\id_X$.
\end{remark}

\begin{definition}[Framed presheaves]\label{def:FrPresheaves}
A framed presheaf with values in a category
$\calC$ is 
a functor $\calF\colon \Fr_*(B)^\mathrm{op}\to \calC$.
\end{definition}

A presheaf $\calF$ on $\Sm_B$ with values in a category $\calC$ 
is a contravariant functor from $\Sm_B$ to $\calC$.
We call $\calF$ \emph{radditive} if the restriction morphisms induce isomorphisms 
$\calF(X\amalg Y)\simeq \calF(X)\times \calF(Y)$, for all $X,Y\in \Sm_B$. 
The direct image along the functor $\Sm_B\to \Fr_*(B)$ takes a framed presheaf 
$\calF$ over $B$ to a presheaf on $\Sm_B$,  
which we denote by the same symbol. 
We say that $\calF$ is \emph{$\A^1$-invariant} if 
the projection $X\times\A^1\to X$ induces an isomorphism 
$\calF(X)\xrightarrow{\cong} \calF(X\times\A^1)$ for all $X\in \Sm_B$.

For a fixed $Y$, 
the $\sigma$-suspension map \eqref{eq:sigmanc} defines a morphism of presheaves on $\Sm_B$ 
(with respect to the first variable).
Moreover, we have $\sigma^n c= \sigma^n_Y \circ c$.
A framed presheaf $\calF$ is called
\emph{quasi-stable} if the induced homomorphism $\sigma^*_X$ is an isomorphism for all $X$.

For a framed presheaf $\calF$, a framed correspondence $c\in \Fr_n(X,Y)$, and $a\in\calF(Y)$, 
we write $c^*(a)\in\calF(X)$ for the inverse image.

\begin{definition}[Linear framed correspondences]
Let $\ZF_*(-,Y)$ be the quotient of the presheaf of free abelian groups on $\Fr_*(-,Y)$ 
modulo the relation $[c]=[c_1]+[c_2]$.
Here $c, c_i\in \Fr_n(X,Y)$, 
$c$ has support $S=S_1\amalg S_2$ and $c_i$ has support $S_i$,
their \'etale neighborhoods coincide, $U=U_i$, 
their corresponding regular functions agree, $\phi=\phi_i$, 
and finally $g=g_i$.  
\end{definition}

We set 
\[
\overline{\ZF}_n(X,Y)
:=\operatorname{Coker}(\ZF_n(X\times\A^1,Y)\xrightarrow{i^*_0-i^*_1}\ZF_n(X,Y)),
\]
where $i^*_j$ is the inverse image along the $j$-section $X\to X\times\A^1$, 
$j=0,1$.

\subsection{Strict homotopy invariance over base fields}

The strict homotopy invariance theorem for $\A^1$-invariant framed presheaves of abelian groups 
is proven over perfect, infinite fields of characteristic different than 2 in \cite{hty-inv}; 
the latter two assumptions are dispensed with in \cite{surj-etale-exc}, 
and the remaining case of imperfect fields is settled in \protect{\cite[Theorems 8.2, 9.32]{nonperfect-SHI}}.
We recall the following equivalent forms of the said theorem. 

\begin{theorem}[\protect{\cite[Theorems 8.2, 9.32]{nonperfect-SHI}}]
\label{citedth:SHI(k)}
For any field $k$, the following hold:

(1) For any $\A^1$-invariant quasi-stable framed radditive presheaf of abelian groups $\calF$, there is a natural equivalence
\[H^*_{\nis}(X,\calF_\nis)\simeq H^*_{\nis}(X\times\A^1,\calF_\nis), X\in \Sm_k,\]
where $\calF_\nis$ denotes the Nisnevich sheafification.

(2) For any $\A^1$-invariant $\calF\in \SH^{\fr,s}(k)$ the presheaf $L_\nis \calF$ is $\A^1$-invariant.

(3) For any $\A^1$-invariant grouplike $\calF\in \Pre^\fr(k)$ the presheaf $L_\nis \calF$ is $\A^1$-invariant.
\end{theorem}

\subsection{The $\infty$-category of tangentially framed correspondences.}
\label{subsect:Corrtfr}

Next, we recall 
the notion of tangentially framed correspondences in \cite{five-authors} which offers an equivalent way of 
viewing framed correspondences in the setting of $\infty$-categories. 

\begin{definition}\label{def:tgfr}
A \emph{tangentially framed correspondence} 
from $X$ to $Y$, where $X,Y\in \Sm_{S}$, 
is defined by the data 
\begin{itemize}
\item[(i)] 
a finite syntomic map $f\colon Z\to X$,
and an arbitrary map $g\colon Z\to Y$
\item[(ii)] 
a trivialization $\mathcal L_{f}\simeq 0$ of the cotangent complex $\mathcal L_{f}$ in the 
$K$-theory $\infty$-groupoid of $Z$
\end{itemize}
For a given $Y$ 
define a presheaf of spaces $h^\ctfr(Y)$ whose value on $X$ is given by 
\[
\begin{tikzcd}
h^\ctfr(Y)(X) = \underset{(f,Z,g)}{\displaystyle{\coprod}} 
\operatorname{eq}\big(*\arrow[r,shift left,"\calL_f"]\arrow[r,shift right,swap,"0"] & K(Z)\big).
\end{tikzcd}
\]
The coproduct involves all triples $(f,Z,g)$ as in (i), 
and the equalizer in (ii) is equivalent to the space of paths between $\mathcal{L}_{f}$ and $0$ 
in the $K$-theory $\infty$-groupoid of $Z$.
While the class of $\mathcal{L}_{f}$ in $K(Z)$ is trivial, 
the choice of a trivialization $[0,1]\rightarrow K(Z)$ is additional data. 
\end{definition}

We refer to \cite{five-authors} for the $\infty$-category $\Corr^\ctfr(B)$ of tangentially framed correspondences, 
see \Cref{subsection:notation}. 
Its objects are smooth schemes over $B$, 
and the mapping space $\mathrm{Map}_{\Corr^\ctfr(B)}(X,Y)$ 
is naturally equivalent to $h^\ctfr(Y)(X)$; 
thus, the presheaf $h^\ctfr(Y)$ is represented by $Y\in\Sm_B$.
A morphism of schemes defines canonically a tangentially framed correspondence, 
which defines a functor of $\infty$-categories $\gamma\colon \Sm_B\to \Corr^\ctfr(B)$. 
There is a naturally induced forgetful functor 
$\gamma_*\colon \Pre^\fr(B)\to \Pre(B)=\Pre(\Sm_B)$,
where $\Pre^\fr(B)$ is a shorthand for $\Pre(\Corr^\ctfr(B))$.
Moreover, $\gamma$ factors through $\Fr_*(B)$ as in 
\begin{equation}\label{eq:SmBtoFrstarBCorrtfrB}
    \Sm_B\to \Fr_*(B)\to \Corr^\ctfr(B).
\end{equation}

We denote the subcategories $\Pre^\fr(B)$ spanned by Nisnevich sheaves, $\A^1$-invariant, 
and group-complete objects by $\Pre^\fr_\nis(B)$, $\Pre^\fr_{\A^1}(B)$, and $\Pre^{\fr,\mathrm{gp}}(B)$,
respectively (see \Cref{subsection:notation}).
We remark that $\Pre^\fr_{\A^1}(B)$, $\Pre^\fr_{\nis}(B)$, and $\Pre^{\fr,\mathrm{gp}}(B)$ 
are reflective subcategories of $\Pre^\fr(B)$.
Finally, 
$\Pre^\fr_{\A^1,\nis}(B)$ denotes the subcategory $\Pre^\fr_{\A^1}(B)\cap \Pre^\fr_{\nis}(B)$ 
of $\A^1$-invariant Nisnevich sheaves.

\begin{definition}
Let $\mathbf H^\fr_{\A^1,\nis}(B)$ denote the homotopy category $\mathrm{h}\Pre^\fr_{\A^1,\nis}(B)$.
Similarly, 
we let $\mathbf H^{\fr,\mathrm{gp}}_{\A^1,\nis}(B)$ denote the homotopy category of the $\infty$-category of 
grouplike $\A^1$-invariant Nisnevich sheaves on $\Corr^\ctfr(B)$.
Moreover, 
we define 
\[\begin{array}{lcl}
\SHfrsAnis(B)&:=&\mathrm{h}\Pre^{\fr}_{\A^1,\nis}(B)[(S^{1})^{-1}],\\
\SHfrstAnis(B)&:=&\mathrm{h}\Pre^{\fr}_{\A^1,\nis}(B)[(S^{1})^{-1},\Gm^{\wedge -1}],\\
\SHfrAnis(B)&:=&\mathrm{h}\Pre^{\fr}_{\A^1,\nis}(B)[(\PP^{1})^{\wedge -1}].
\end{array}\]
The same formulas apply to any Grothendieck topology on $\Sm_S$.
\end{definition}

\begin{theorem}[\protect{\cite[Theorem 18]{Hoyois-localization}}]
\label{cth:SHfreqSH}
The forgetful functor
$\gamma_*\colon \Pre^\fr(B)\to \Pre(B)$ induces an equivalence 
$$\SH^{\fr}_{\A^1,\nis}(B)\simeq \SH_{\A^1,\nis}(B).$$
\end{theorem}

\begin{corollary}\label{lm:SHAnistoSHfrsAnis}
The functors
\begin{equation}\label{eq:SHAnistoSHfrsAnis}
\SH_{\A^1,\nis}(B)\xrightarrow{\gamma^*}
\SH^\fr_{\A^1,\nis}(B)\xrightarrow{\Omega^{\infty,\fr}_{\Gm}}
\SH^{\fr,s}_{\A^1,\nis}(B)
\end{equation}
participate in an isomorphism 
\begin{equation}
\label{eq:piSHAniscalFsimeqpiSHfrsAniscalF}\pi_{n,0}(\calF)
\cong
\pi_n (\Omega^{\infty,\fr}_{\Gm}\gamma^*(\calF)).
\end{equation}
\end{corollary}
\begin{proof}
By \Cref{cth:SHfreqSH}
there is an equivalence $\calF\simeq\gamma_*\gamma^*\calF$, 
and therefore
\begin{equation*}
\label{eq:piSHAniscalFsimeqpiSHfrsAniscalFproof}
\begin{array}{lcl}
\pi_{n,0}(\calF)
&\cong& 
\pi_{n,0}(\gamma_*(\gamma^*\calF))\\
&\cong& 
\pi_n(\Omega^{\infty}_{\Gm}\gamma_*(\gamma^*\calF))
\\
&\cong&
\pi_n(\gamma_*\Omega^{\infty,\fr}_{\Gm}\gamma^*(\calF))
\\
&\cong& 
\pi_n(\Omega^{\infty,\fr}_{\Gm}\gamma^*(\calF))
.
\end{array}\end{equation*}
\end{proof}
\begin{remark}
    More generally, by the same proof, one obtains an isomorphism
    \[\pi_{n,l}(\calF)\cong \pi_n(\Omega^{\infty,\fr}_{\Gm}\gamma^*(\Gm^{\wedge -l}\wedge\calF)).\]
    The case of weight zero in \Cref{lm:SHAnistoSHfrsAnis} suffices for our purposes.
\end{remark}

\begin{definition}
The subcategory of \emph{quasi-stable framed} objects in $\SHs(B)$ is the 
essential image of the subcategory spanned by quasi-stable framed presheaves of $S^1$-spectra 
under the direct image functor $\SHfrs(B)\to\SHs(B)$ induced by \eqref{eq:SmBtoFrstarB}.
\end{definition}

\begin{remark}
The essential image of $\SHfrs(B)\to\SHs(B)$ is contained in the essential image of 
$\SHs(\Fr_*(B))\to\SHs(B)$, see \eqref{eq:SmBtoFrstarBCorrtfrB}.
In fact,  
any object in the essential image of $\SHfrs(B)\to\SHs(B)$ is quasi-stable framed 
(for each $X\in\Sm_B$ the image of $\sigma_X$ along $\Fr_*(B)\to\Corr^\fr(B)$
belongs to the same connected component as the identity morphism $\mathrm{id}_X$).
\end{remark}

\section{Cousin and $\tf$-Cousin complexes}\label{sect:CousbitfComplexes}

Throughout this section we consider the category $\TCat(B)$ of presheaves on $\EssSm_B$ with values in a 
stable $\infty$-category $\TCat$. 
Let $\calF\mapsto \calF[1]$ be the suspension functor on $\TCat$.
We write $\hTCat$ for the (additive) homotopy category of $\TCat$, 
and $\KC$, $\biKC$ for the corresponding categories of chain complexes and chain bicomplexes.
The value of $\calF$ on $X\in\EssSm_B$ is denoted by $\calF(X)\in\TCat$ and we write 
$\hF(X)\in \hTCat$ for the image in $\hTCat$.
Furthermore, we denote by $\hTCat(B)$ the additive category of presheaves with values in $\hTCat$, 
and by $\hF\in \hTCat(B)$ the presheaf induced by $\calF$.

Recall that $\calF\in\TCat(B)$ is local with respect to a cd-structure $\tau$ on $\EssSm_B$ in the sense of \cite{VV:cd} 
if it takes each distinguished square to a Cartesian square in $\TCat$, 
\cite[Definition 6.2.2.6]{Lurie}, and
see \cite[Definition 1.4.1]{zbMATH01573275}.
Denote by $L_\tau\colon \TCat(B)\to \TCat(B)$ the respective localization functor that lands in the subcategory of local objects.

\begin{example}\label{ex:TCatisspectra}
    If $\TCat$ is the $\infty$-category of spectra, then $\hTCat=\SHtop$. 
    Any presheaf of spectra on $\Sm_B$ defines a presheaf on $\EssSm_B$ by continuity. 
    We denote both presheaves by the same symbol $\calF\in\SHs(B)$. 
\end{example}

\subsection{Cousin complexes}\label{sect:Cousincomplex}

\begin{definition}
\label{def:FYUFxU}
For $U\in\Et_X$ and any closed subscheme $Y\not\hookrightarrow U$, we set
\begin{equation*}
\calF_Y(U):=\fib( \calF(U)\to \calF(U-Y) ).
\end{equation*}
If $x\in U$, then
$\calF_x(U):=\calF_{x}(U_{(x)})$,
where $U_{(x)}$ is the local scheme of $U$ at $x$.
If $Y$ is not a subscheme of $U$, 
we set $\calF_Y(U):=\calF_{Y\cap U}(U)$.
\end{definition}

\Cref{def:FYUFxU} implies there is an
equivalence 
\[\calF_{Y_1}(U)\calsimeq \fib(\calF_Y(U)\to \calF_{Y-Y_1}(U-Y_1))\]
for any $U\in\Et_X$,
closed immersions $Y_1\not\hookrightarrow Y\not\hookrightarrow U$, 
and $\calF\in \TCat(B)$.
Thus the stability of $\TCat$ provides a
morphism
\begin{equation}\label{eq:partialYY1}
\partial^Y_{Y_1}
\colon 
\calF_{Y-Y_1}(U-Y_1)
\to \calF_{Y_1}(U)[1].
\end{equation}

\begin{lemma}\label{lm:Y2Y1YUtrivdifcomposite}
Let
$Y_2\not\hookrightarrow Y_1 \not\hookrightarrow Y\not\hookrightarrow U$ be closed immersions.
Then 
for any Nisnevich local $\calF\in \TCat(B)$,
the composite morphism 
\begin{equation}\label{eq:partYY1partY1Y2}
\hF_{Y}(U-Y_1)\xrightarrow{\partial^{Y}_{Y_1}} 
\hF_{Y_1-Y_2}(U-Y_2)[1]\xrightarrow{\partial^{Y_1}_{Y_2}[1]}
\hF_{Y_2}(U)[2]
\end{equation}
is trivial in $\hTCat$.
\end{lemma}

\begin{proof}
There is the commutative diagram in $\TCat$
\begin{equation}
\xymatrix{
\calF_{Y}(U-Y_1)\ar[r]\ar[dr]& \calF_{Y_1-Y_2}(U-Y_2)[1]\ar[r]& \calF_{Y_2}(U)[2]\\
&\calF_{Y_1}(U)[1]\ar[u]\ar[ru]&
}
\label{eq:TcommdiagrFYU}
\end{equation}
and the morphisms in \eqref{eq:partYY1partY1Y2} fit into the commutative diagram in 
$\hTCat$ induced by \eqref{eq:TcommdiagrFYU}.
Since the right-most diagonal morphism in \eqref{eq:TcommdiagrFYU} corresponds to the trivial morphism in $\hTCat$, 
both the composite morphisms in \eqref{eq:TcommdiagrFYU} are trivially mapped to $\hTCat$.
\end{proof}

Suppose $x,x^\prime\in X$ are points such that $x$ is contained in the 
closure $Y^\prime$ of ${x^\prime}$ in $X$, 
and $\codim_{X} x=\codim_{X} x^\prime+1$.
Then the local scheme $(Y^\prime)_{(x)}$ has dimension one and $(Y^\prime)_{(x)}-x=x^\prime$.
Thus for any Nisnevich local object $\calF\in \TCat(B)$, 
we have an equivalence
\[
\calF_{(Y^\prime)_{(x)}-x}(X_{(x)}-x)
\xrightarrow{\calsimeq}
\calF_{x^\prime}(X) 
,\]
and hence a composite morphism
\begin{equation}\label{eq:partialxprimex}
\partial^{x^\prime}_x\colon 
\calF_{x^\prime}(X) \calsimeq \calF_{(Y^\prime)_{(x)}-x}(X_{(x)}-x)\to 
\calF_{x}(X)[1],
\end{equation}
obtained from \eqref{eq:partialYY1}.

\Cref{lm:Y2Y1YUtrivdifcomposite} implies the following corollary.

\begin{corollary}\label{cor:CoustfCousdifferntialsquaretrivial}
    For $0\leq c\leq \dim X$ the composite
    \[
    \bigoplus_{x\in X^{(c)}} \hF_{x}(X)[c]\to 
    \bigoplus_{x\in X^{(c+1)}} \hF_{x}(X)[c+1] \to 
    \bigoplus_{x\in X^{(c+2)}} \hF_{x}(X)[c+2]
    \]
    defined by \eqref{eq:partialxprimex}
    is trivial in $\hTCat$.
\end{corollary}

\begin{definition}\label{def:Cou}
For any Nisnevich local presheaf $\calF\in\TCat(B)$
and $X\in \EssSm_B$, 
\eqref{eq:partialxprimex} provides the sequence of morphisms in $\hTCat$
\begin{equation}\label{eq:seq:res:CXF}
\bigoplus_{x\in X^{(0)}} \hF_{x}(X) \to \dots \to 
\bigoplus_{x\in X^{(c)}} \hF_{x}(X)[c]\to
\bigoplus_{x\in X^{(c+1)}} \hF_{x}(X)[c+1] \to \dots \to 
\bigoplus_{x\in X^{(\dim X)}} \hF_{x}(X)[\dim B]
.\end{equation}
According to \Cref{cor:CoustfCousdifferntialsquaretrivial}, 
the sequence \eqref{eq:seq:res:CXF} defines a differential complex which we denote by
\[
\Cou(X,\calF)\in\KC.
\] 
\end{definition}

\begin{remark}
The morphisms \eqref{eq:partialYY1} are functorial with respect to arbitrary morphisms of schemes. 
The sequence \eqref{eq:seq:res:CXF} is functorial with respect to \'etale morphisms in $\EssSm_B$ since 
\'etale morphisms preserve codimension of points.
Thus $\Cou(-,\calF)$ defines a presheaf on the small \'etale site $\Et_X$ with values in $\KC$ for any $X\in\Sm_B$.
\end{remark}

For $X$ and $\calF$ as above, 
the morphism 
\begin{equation}\label{eq:calFXtocalFXzero}
\calF(X)\to \calF(X^{(0)})
\end{equation}
induces a morphism in $\KC$
\begin{equation}\label{eq:FtoCou}
\hF(X)\to \Cou(X,\calF), 
\end{equation}
where the left side is viewed as a complex concentrated in degree zero.

\begin{definition}
\label{def:ovCou}
For $X\in\EssSm_B$ and a Nisnevich local $\calF\in\TCat(B)$,
\emph{the Cousin complex} 
\[
\ovCou(X,\calF):= 
\operatorname{Cone}(\hF(X)\to 
\Cou(X,\calF))
\]
is the cone of the morphism in \eqref{eq:FtoCou} in $\KC$.
\end{definition}

Generalizing \eqref{eq:calFXtocalFXzero} and \eqref{eq:FtoCou}, 
for any $\calF\in\TCat(B)$,
the morphisms \[\calF(X)\to \calF(X^{(0)})\xrightarrow{\calsimeq} L_\nis\calF(X^{(0)})\]
induce the morphism in $\KC$
\begin{equation}\label{eq:FtoCouLnis}
\hF(X)\to \Cou(X,L_{\nis}\calF).
\end{equation}
We can, therefore, make the following generalization of \Cref{def:ovCou}:

\begin{definition}\label{def:calFTCatB:Cou}
For any $\calF\in\TCat(B)$, and any $X\in\EssSm_B$,
we define the associated Cousin complexes as the objects in $\KC$ given by 
\begin{align*}
\Cou(X,\calF) &:= \Cou(X,L_{\nis}\calF)\\
\ovCou(X,\calF) &:= \operatorname{Cone}(\hF(X)\to \Cou(X,L_{\nis}\calF)).
\end{align*}
\end{definition}

\subsection{$\tf$-Cousin complex}\label{sect:tfCousincomplexdefinition}

In this section, 
we define the $\tf$-Cousin complex associated to any $X\in\EssSm_B$ and 
$\tf$-local presheaf $\calF\in\TCat(B)$.
We begin by recalling the $\tf$-topology introduced in \cite[Definition 3.1]{DKO:SHISpecZ}.

\begin{definition}[$\tf$-topology]
The $\tf$-topology on $\Sm_B$ is the cd-topology generated by the Nisnevich squares 
\[\xymatrix{\widetilde{U} \ar[d]\ar[r]& \widetilde{X}\ar[d]\\ U \ar[r]& X}\]
such that $X\setminus U\cong X\times_B Z$, for some closed subscheme $Z$ in $B$.
\end{definition}

\begin{definition}\label{def:calFZU}
For $U\in\Et_X$ and any closed subscheme $Z\not\hookrightarrow B$, we set
\begin{equation*}
\calF_Z(U):=\calF_{U\times_B Z}(U).
\end{equation*}
If $z\in B$, we write $B_{(z)}$ for the local scheme of $B$ at $z$, and set
\begin{equation*}
\calF_z(U):=\calF_{z}(U\times_B B_{(z)}).
\end{equation*}
\end{definition}

Similarly to 
\eqref{eq:partialxprimex}, 
for any $\tf$-local object $\calF\in\TCat(B)$,
there are morphisms
\begin{equation}\label{eq:partialzprimez}
\partial^{z^\prime}_z\colon 
\calF_{z^\prime}(X) \calsimeq \calF_{(Z^\prime)_{(z)}-z}(X\times_B (B_{(z)}-z))\to 
\calF_{z}(X\times_B B_{(z)})[1]\calsimeq \calF_{z}(X)[1]
\end{equation}
for each pair of points $z,z^\prime\in B$  
such that $z$ is contained in the closure $Z^\prime$ of ${z^\prime}$ in $B$, 
and $\codim_{B} z=\codim_{B} z^\prime+1$.

\begin{definition}\label{def:tfCou}
For any $\tf$-local $\calF\in \TCat(B)$ and $X\in \EssSm_B$,
\eqref{eq:partialzprimez} defines the following sequence of morphisms
in $\hTCat$
\begin{equation}\label{eq:seq:res:CtfXF}
\bigoplus_{z\in B^{(0)}} \hF_{z}(X) \to \dots \to 
\bigoplus_{z\in B^{(i)}} \hF_{z}(X)[i]\to 
\bigoplus_{z\in B^{(i+1)}} \hF_{z}(X)[i+1] \to \dots \to 
\bigoplus_{z\in B^{(\dim B)}} \hF_{z}(X)[\dim B].
\end{equation}
Similarly to \Cref{def:Cou}, 
the sequence \eqref{eq:seq:res:CtfXF} defines a differential complex 
\[\Cou_\tf(X,\calF)\in\KC.\]
\end{definition}

\begin{remark}
The sequence \eqref{eq:seq:res:CtfXF} is functorial with respect to arbitrary morphisms in $\EssSm_B$.
Thus $\Cou_\tf(-,\calF)$ defines a presheaf on $\EssSm_B$ with values in $\KC$.
This does not hold when we replace $\EssSm_B$ with $\Sch_B$. 
\end{remark}

For every $\tf$-local $\calF\in\TCat(B)$ and $X\in \EssSm_B$,
the morphism \begin{equation}\label{eq:calFXcalFXBzero}\calF(X)\to \calF(X\times_B B^{(0)})\end{equation}
induces a morphism in $\KC$
\begin{equation}\label{eq:FtotfCou}
\hF(X)\to \Cou_\tf(X,\calF).
\end{equation}
The left side is viewed as a complex concentrated in degree zero.

\begin{definition}
\label{def:ovCoutf}
The $\tf$-Cousin complex of a $\tf$-local $\calF\in\TCat(B)$ and $X\in\EssSm_B$
is given by the cone 
\[\ovCou_\tf(X,\calF) = \operatorname{Cone}(\hF(X)\to \Cou_\tf(X,\calF))\]
of the morphism in \eqref{eq:FtotfCou}.
\end{definition}

\begin{remark}\label{rem:CoutfCou} 
For any $\calF\in \TCat(B)$ and $X\in\EssSm_B$, 
we have
\[
\Cou(X,\calF)\cong\Cou_\tf(X,p^*\calF),
\quad
\Cou_\tf(X,\calF)\cong \Cou(B,p_*(p^*\calF)),
\] 
where $p^*\colon\TCat(B)\to\TCat(X)$ and $p_*\colon\TCat(X)\to\TCat(B)$ are the inverse and direct image functors along the canonical morphism $p\colon X\to B$.
\end{remark}

For any $\calF\in\TCat(B)$,
\eqref{eq:FtotfCou} gives a morphism
\begin{equation*}
L_\tf\hF(X)\to \Cou_\tf(X,L_{\tf}\calF).
\end{equation*}
Composing with the canonical morphism $\hF(X)\to L_\tf \hF(X)$,
we get a morphism in $\KC$
\begin{equation}\label{eq:FtotfCouLtf}
\hF(X)\to \Cou_\tf(X,L_{\tf}\calF).
\end{equation}

\begin{definition}
\label{def:calFTCatB:tfCou}
For every $\calF\in\TCat(B)$ and $X\in\EssSm_B$,
we define objects of $\KC$ by
\begin{align*}
\Cou_\tf(X,\calF) &:= \Cou_\tf(X,L_{\tf}\calF)\\
\ovCou_\tf(X,\calF) &:= \operatorname{Cone}(\calF(X)\to\Cou_\tf(X,L_{\tf}\calF)).
\end{align*}
\end{definition}

For any Nisnevich local $\calF$, 
we will in \Cref{def:CoutftoCoubi_AND_CoutftoCou} construct in a canonical morphism 
\begin{equation}\label{eq:CoutftoCou_init}\Cou_\tf(X,\calF)\to \Cou(X,\calF)\end{equation} in $\KC$
and a commutative triangle 
\begin{equation*}
\label{eq:FCoutfCou}\xymatrix{
&\calF(X)\ar[ld]\ar[rd]&\\
\Cou_\tf(X,\calF)\ar[rr]&&\Cou(X,\calF)
}
\end{equation*}
To construct \eqref{eq:CoutftoCou_init},  
we will define a bicomplex $\Coubi(X,\calF)$ whose totalization is isomorphic to $\Cou(X,\calF)$, 
see \Cref{def:CoubiXF} and \Cref{lm:bicompexsimeqlcomlex}. 
Moreover, we will make use of the naturally induced commutative triangle in $\TCat$ 
\begin{equation*}
\label{eq:FCoutfzeroCouzero}\xymatrix{
&\calF(X)\ar[ld]\ar[rd]&\\
\bigoplus_{z\in B^{(0)}}\calF(X\times_B z)\ar[rr]&&\bigoplus_{x\in X^{(0)}}\calF(x)
}
\end{equation*}
and the adjunction 
\begin{equation*}\label{adj:SHtopKgeq0SHtop}
\TCat \rightleftarrows \Kom_{\geq 0}(\TCat)
\end{equation*}
(the right adjoint picks out the degree zero term of a chain complex, 
while the left adjoint takes values in chain complexes concentrated in degree zero).

\subsection{Cousin bicomplexes}\label{sect:CoubiXF}

\begin{definition}\label{def:CzbulletXF}
Suppose that $\calF\in \TCat(B)$ is Nisnevich local, $X\in \EssSm_B$, and $z\in B$. 
We set 
\begin{align*}
\Cou_{z}(X,\calF) :=
&C^{z}_0\to 
\dots \to
C^{z}_j[j]
\to 
C^{z}_{j+1}[j+1] \to 
\dots \to 
C^{z}_{\dim X_z}[\dim X_z]
,\end{align*}
where
\[
C^{z}_{j} :=
\bigoplus\limits_{
x\in X_z^{(j)}
}
\hF_{x}(X).
\]
\end{definition}

\begin{remark}
Recall that $B_{(z)}$ denotes the local scheme at $z\in B$.
There is a closed embedding 
$i_z\colon z\not\hookrightarrow B_{(z)}$ and an open immersion
$j_z\colon B_{(z)}\to B$.
For any $\calF\in\SHs_{\A^1,\nis}(B)$,
there is an isomorphism in $\Kom(\SHtop)$
\[\Cou_{z}(X,\calF) \cong \Cou(X_z, i_z^! j_z^* \calF).\]
\end{remark}

Suppose $z, z^\prime\in B$ are points such that $\overline{z}\supset \overline{z^\prime}$ and 
\begin{equation}
\label{eq:zzprimeiinciB} 
   \codim_B \overline{z}=\codim_B \overline{z^\prime}+1.
\end{equation}
Here, $\overline{z}$ and $\overline{z^\prime}$ denote the minimal closed subschemes in $B$ containing
$z$ and $z^\prime$, 
respectively.
If $\calF\in \TCat(B)$ is Nisnevich local and $X\in \EssSm_B$, 
we aim to define a morphism of complexes
\begin{equation}
\label{eq:CztoCzprime}
\Cou_{z^\prime}(X,\calF)\to \Cou_{z}(X,\calF)[1].
\end{equation}

\begin{lemma}\label{lm:yyprimezzprime}
Given $z$ and $z^\prime$ as above, 
let $y\in X_z$,
$y^\prime\in \overline{y}_{z^\prime}$ and suppose that $\codim_{X}y^\prime=\codim_{X}(y)+1$.
Then we have $\codim_{X_z}y = \codim_{X_{z^\prime}}y^\prime$.
\end{lemma}

\begin{proof}
Consider the $1$-dimensional scheme $Z=\overline{z}_{z^\prime}$
and set $Y=\overline{y}\times_{B}\overline{z}_{z^\prime}$.
Since $y^\prime\in \overline{y}_{z^\prime}$, $y^\prime$ is contained in $Y_{z^\prime}$. 
Moreover, it follows by the assumption on codimensions of $y$ and $y^\prime$ that $y^\prime$ is a generic point of $Y_{z^\prime}$. 
Since $Y$ is irreducible and $Y_z=Y\times_B z\neq\emptyset$, 
by \cite[Corollary A.3]{nonperfect-SHI} it follows that $Y$ is equidimensional over $Z$.
\end{proof}

Consider points $y^\prime\in X_{z^\prime}$ and $y\in \overline{y}_{z}$.
In view of \Cref{lm:yyprimezzprime} the construction of \eqref{eq:partialxprimex} gives a morphism
\[
\calF_{y^\prime}(X)\to \calF_{y}(X)[1].
\]
Summing over the set of points of codimension $j$,
we get a morphism
\begin{equation}\label{eq:vj}
v^{z,z^\prime}_j \colon 
C^{z^\prime}_{j}
\to 
C^{z}_{j}[1].
\end{equation}

\begin{lemma}\label{cor:ddtwoschemes}
Let $Y_1^\prime\not\hookrightarrow Y_1\cup Y^\prime \not\hookrightarrow Y \not\hookrightarrow X$ be closed immersions, $Y_1^\prime=Y_1\cap Y^\prime$.
Then
the square 
\[\xymatrix{
\hF_{Y}(X-Y_1\cup Y^\prime)\ar[r]^-{-\partial^Y_{Y^\prime}}\ar[d]_{\partial^Y_{Y_1}}& 
\bigoplus \hF_{Y^\prime-Y_1^\prime}(X-Y)[1]\ar[d]^{\partial^{Y^\prime}_{Y^\prime_1}[1]}\\ 
\hF_{Y_1-Y_1^\prime}(X-Y)[1]\ar[r]^-{\partial^{Y_1}_{Y^\prime_1}[1]} &
\hF_{Y_1^\prime}(X)[2]
}\]
commutes in $\hTCat$.
\end{lemma}

\begin{proof}
    Consider the morphisms
    \[\hF_{Y}(X-Y_1\cup Y^\prime)\xrightarrow{\partial^{X}_{Y_1\cup Y^\prime}}\hF_{Y_1\cup Y^\prime}(X-Y_1^\prime)[1]\xrightarrow{\partial^{Y_1\cup Y^\prime}_{Y^\prime_1}[1]}\hF_{Y_1^\prime}(X)[2].\]
    Applying \Cref{lm:Y2Y1YUtrivdifcomposite} to the composite
    $Y_1^\prime\to Y_1\cup Y^\prime\to X\xrightarrow{\cong}X$, 
    it follows that 
    \[\partial^{Y^\prime}_{Y^\prime_1}[1]\partial^Y_{Y^\prime}+\partial^{Y_1}_{Y^\prime_1}[1]\partial^Y_{Y_1}=
    \partial^{Y_1\cup Y^\prime}_{Y^\prime_1}[1]\partial^{X}_{Y_1\cup Y^\prime}
    =0.\]
\end{proof}

\begin{corollary}\label{cor:yypovyovyp}
    Suppose $\calF\in \TCat(B)$ is Nisnevich local, $X\in \EssSm_B$,
    and $z, z^\prime$ are as in \eqref{eq:zzprimeiinciB}.
    Then the square
\[\xymatrix{
C^{z^\prime}_j
\ar[r]^-{-v_j}\ar[d] &
C^{z}_j[1]\ar[d]
\\ 
C^{z^\prime}_{j+1}
\ar[r]^-{v_{j+1}}&
C^{z}_{j+1}[1]
}\]
    commutes in $\KC$.
\end{corollary}

\begin{lemma}
Suppose $\calF\in \TCat(B)$ is Nisnevich local and $X\in \EssSm_B$.
The morphisms $(-1)^{j}v_j$ from \eqref{eq:vj} for $j= 0, \dots ,\dim_B X$
define morphism of complexes \eqref{eq:CztoCzprime}.
\end{lemma}
\begin{proof}
The claim holds by \Cref{cor:yypovyovyp}.
\end{proof}

Similarly to \Cref{cor:CoustfCousdifferntialsquaretrivial}, 
the following result follows from \Cref{lm:Y2Y1YUtrivdifcomposite}.

\begin{corollary}
\label{cor:rowsdifferntialsquaretrivial}
 Suppose $\calF\in \TCat(B)$ is Nisnevich local and $X\in \EssSm_B$.
    For each $i$, $j$, 
    the composite
    \[
    \bigoplus\limits_{z^{\prime\prime}\in B^{(i)}}C^{z^{\prime\prime}}_j\to
    \bigoplus\limits_{z^\prime\in B^{(i+1)}}C^{z^\prime}_j[1]\to
    \bigoplus\limits_{z\in B^{(i+2)}}C^{z}_j[2]
    \]
    is trivial.
\end{corollary}

For all $0\leq j\leq\dim_B X$, $0\leq i\leq\dim B$,
we obtain a morphism in $\KC$ 
\begin{equation}
\label{eq:vij}v^i_j = \bigoplus_{z,z^\prime} v^{z,z^\prime}_{j}.
\end{equation}
Here $z,z^\prime\in B$ are points as in \eqref{eq:zzprimeiinciB}, 
and $v^{z,z^\prime}_j$ is defined in \eqref{eq:vj}.

\begin{corollary}
Suppose $\calF\in \TCat(B)$ is Nisnevich local and $X\in \EssSm_B$.
Then the composite $v^{i}_{j} v^{i+1}_j$ is trivial.
\end{corollary}

\begin{definition}
\label{def:CoubiXF}
If $\calF\in \TCat(B)$ is Nisnevich local and $X\in \EssSm_B$, we set
\begin{align*}
\Coubi(X,\calF)=
&\bigoplus_{z\in B^{(0)}} \Cou_{z}(X,\calF) \to \dots \to 
\bigoplus_{z\in B^{(i)}} \Cou_{z}(X,\calF)[i]\to \\ 
&\to\bigoplus_{z\in B^{(i+1)}} \Cou_{z}(X,\calF)[i+1] \to \dots \to
\bigoplus_{z\in B^{(\dim B)}} \Cou_z(X,\calF)[\dim B]
\end{align*}
The columns in this bicomplex are sums of shifts of the complexes $\Cou_{z}(X,\calF)$, for $z\in B$, 
and the horizontal morphisms are determined by \eqref{eq:vij}.
\end{definition}

\subsection{Relating the various Cousin complexes}
\label{sect:FXCoutfXFCoubiXF}

We consider the totalization functor given by taking coproducts
\begin{equation}
\label{eq:TotbiKomSHtoptoKomSHtop}
\Tot\colon \biKC\to \KC.
\end{equation}
 
\begin{lemma}\label{lm:bicompexsimeqlcomlex}
For any Nisnevich local $\calF\in \TCat(B)$ and $X\in \EssSm_B$, 
there is a natural isomorphism 
\begin{equation}
\label{eq:TotbiCousimeqCou}
\Tot(\Coubi(X,\calF))\cong \Cou(X,\calF).
\end{equation}
\end{lemma}

\begin{proof}
The terms in \eqref{eq:TotbiCousimeqCou} are identified via the equivalences
\begin{equation}
\label{eq:diagonalSupbiCousinComplex}
\bigoplus\limits_{{i+j=c}, \,{z\in B^{(i)}, \,y\in X_z^{(j)} } } \calF_{y}(X)[j+i]  
\calsimeq 
\bigoplus\limits_{x\in X^{(c)}} \calF_{x}(X)[c]
\end{equation}
The differentials in $\Tot(\Coubi(X,\calF))$ are defined by the sums of the morphisms
\[\begin{array}{llll}
\calF_{y^\prime}(X)[j+i] \to \calF_{y}(X)[(j+1)+i], 
&y,y^\prime\in X_z, &\codim_B z=i, \\
& \codim_{X_z} y=j+1, &\codim_{X_z} y^\prime=j,\\
\calF_{y^\prime}(X)[j+i] \to \calF_{y}(X)[j+(i+1)], &y\in X_z, y^\prime\in X_{z^\prime}, & \codim_{X_z} y=\codim_{X_z} y^\prime=j, \\
&\codim_B z=i+1,&\codim_B z^\prime=i, 
\end{array}\]
where in both rows $y\in \overline {y^\prime}$, and $y,y^\prime\in X$, and $z\in B$, and in the second row $z\in \overline {z^\prime}$, for $z^\prime\in B$.

On the other hand, the differentials in $\Cou(X,\calF)$ are given by the sums of the morphisms
\[
\calF_{x^\prime}(X)[c] \to \calF_{x}(X)[c+1], x\in \overline{x^\prime}, \codim_{X} x=c+1, \codim_{X} x^\prime=c.
\]
Since the equivalences \eqref{eq:diagonalSupbiCousinComplex} commute with the above differentials, 
our claim follows.
\end{proof}

For all $z\in B$, there is a canonically induced morphism 
\begin{equation}\label{eq:FzCouzF}
\hF_{z}(X)\to \Cou_{z}(X,\calF)
\end{equation}
along with a commutative diagram in $\hTCat$
\[\xymatrix{
\hF_{z}(X)\ar[d]\ar[r]& \Cou_{z}(X,\calF)\ar[d]\\
\hF_{z^\prime}(X)[1]\ar[r]& \Cou_{z^\prime}(X,\calF)[1],
}\]
where the vertical morphisms are \eqref{eq:partialzprimez} and \eqref{eq:CztoCzprime}.
Thus, we have a canonical map of bicomplexes
\[
\Cou_\tf(X,\calF)\to \Coubi(X,\calF), 
\]
which, in view of \Cref{lm:bicompexsimeqlcomlex}, 
yields \eqref{eq:CoutftoCou_init}.

\begin{definition}
\label{def:Nis:CoutftoCoubi_AND_CoutftoCou}
If $\calF\in\TCat(B)$ is Nisnevich local, consider the canonical morphism 
\begin{equation}
\label{eq:Nis:CoutftoCoubi}
\Cou_\tf(X,\calF)\to \Coubi(X,\calF)
\end{equation}
where $\Cou_\tf(X,\calF)$ is viewed as a bicomplex
concentrated in the zeroth row. 
Let 
\begin{equation*}
\label{eq:Nis:CoutftoCou}
\Cou_\tf(X,\calF)\to 
\Cou(X,\calF)
\end{equation*} 
be the morphism in $\KC$ obtained by applying \eqref{eq:TotbiKomSHtoptoKomSHtop} to \eqref{eq:Nis:CoutftoCoubi},  
see the isomorphism \eqref{eq:TotbiCousimeqCou}.
\end{definition}

\begin{remark}
If $\calF\in \TCat(B)$ is Nisnevich local, 
$X\in \EssSm_B$, 
then in $\biKC$ there is an isomorphism $\Coubi(X,\calF)\cong \Cou_\tf(X,\Cou(\calF))$, 
where a similar formula gives the right side as in \Cref{def:tfCou}
applied to the presheaf $\Cou(\calF)=\Cou(-,\calF)$ on the small Zariski site on $X$.
\end{remark}

We generalize \Cref{def:Nis:CoutftoCoubi_AND_CoutftoCou} to any $\calF\in\TCat(B)$ as follows.
Given $\calF\in\TCat(B)$,
for all $z\in B$,
consider 
the canonical morphisms $L_{\tf}\calF_{z}(X)\to L_{\nis}\calF_{z}(X)$ in $\TCat$ and the sequence of morphisms in $\KC$
\begin{equation}\label{eq:FzCouzFLtfNis}
L_{\tf}\calF_{z}(X)\to L_{\nis}\calF_{z}(X)\to \Cou_{z}(X,L_{\nis}\calF)
\end{equation}
where two left-side terms are considered complexes concentrated in degree zero.

\begin{definition}\label{def:CoutftoCoubi_AND_CoutftoCou}
For any $\calF\in\TCat(B)$, consider the canonical morphism 
\begin{equation}\label{eq:CoutftoCoubi}
\Cou_\tf(X,L_\tf\calF)\to \Coubi(X,L_{\nis}\calF)
\end{equation}
where the left side is considered as an object in $\biKom(\SHtop)$ 
concentrated in the zeroth row. 
Then, we define the morphism 
\begin{equation}\label{eq:CoutftoCou}\Cou_\tf(X,L_{\tf}\calF)\to \Cou(X,L_{\nis}\calF)\end{equation} in $\KC$
applying the functor \eqref{eq:TotbiKomSHtoptoKomSHtop}
in view of the isomorphism \eqref{eq:TotbiCousimeqCou}.
The morphism \eqref{eq:CoutftoCou} defines the morphism \eqref{eq:CoutftoCou_init}
for any $\calF\in\TCat(B)$ with respect to \Cref{def:calFTCatB:Cou,def:calFTCatB:tfCou}
\end{definition}

\begin{definition}
\label{def:piCou}
We define
$\pi\Cou(X,\calF)=\bigoplus_{l\in\bbZ}\pi_l\Cou(X,\calF)$,
and similarly for $\pi\ovCou(X,\calF)$, $\pi\Cou_\tf(X,\calF)$, $\pi\ovCou_\tf(X,\calF)$, and $\pi\Coubi(X,\calF)$.
\end{definition}

\section{Extended compactified framed correspondences}\label{sect:FocCompCor}

In this section, 
we introduce some variations of Voevodsky's notion of framed correspondences, as in 
\Cref{def:VoevFr}.

\subsubsection{Compactified framed correspondences}

\begin{definition}[Affine correspondences]\label{def:affcorr}
Let $X$ be an affine $B$-scheme,
$Y$ be a $B$-scheme,
and $d\ge0$ be an integer.
A \emph{$d$-dimensional affine correspondence} from $X$ to $Y$ over $B$
is an affine $X$-scheme $S$ of relative dimension $d$ and a morphism of schemes $S\to Y$,
so there is a span 
\begin{equation}
\label{eq:spanXSY}
X\leftarrow S\rightarrow Y.
\end{equation}
The $X$-scheme $S$ is called the \emph{support} of the correspondence.
\end{definition}

\begin{definition}[Framings]\label{def:reldimFr}
Let $n\ge d\ge0$ be integers.
A \emph{level $n$ framing} of 
a given $d$-dimensional affine correspondence \eqref{eq:spanXSY} from $X\in\Aff_B$ to $Y\in\Sch_B$ over $B$
is the data of
\begin{itemize}
\item[(1)]
a closed embedding of $X$-schemes $S\not\hookrightarrow \A^n_X$
\item[(2)]
an $(n-d)$-tuple of regular functions $\varphi\in \mathcal O((\A^n_X)^{\oplus n-d}$,
such that 
the image of $S$ in $\A^n_X$
is a clopen subscheme
of
the vanishing locus $Z(\varphi)$,
\item[(3)]
a morphism of $B$-schemes $g\colon (\A^n_X)^h_{S}\to Y$,
defined on the henselization $(\A^n_X)^h_{S}$ 
for the image of $S$, 
lifting $S\to Y$, i.e., the triangle
\[\xymatrix{(\A^n_X)^h_{S}\ar[r] & Y\\ S\ar@{^(->}[u]\ar[ru]}\]
commutes.
\end{itemize}
A $d$-dimensional affine correspondence equipped with a level $n$ framing is also called a 
\emph{$d$-dimensional level $n$ framed correspondence} and denoted more succinctly as $\Phi=(S,\varphi,g)$. 
We let $\Fr_n^d(X,Y)$ denote the set of $d$-dimensional level $n$ framed correspondences from $X$ to $Y$.
\end{definition}

\begin{example}
The notion of \emph{quasi-finite framed correspondences of level $n$} is defined similarly by requiring 
that $f\colon S\to X$ is quasi-finite, 
see \cite[Definition 3.2]{rel-mot-sphere}.
By definition, $0$-dimensional level $n$ framed correspondences coincide with quasi-finite framed correspondences of level $n$.
\end{example}

\begin{definition}[Compactified correspondences]\label{def:compactifiedcorr}
A compactification of some \emph{$d$-dimensional affine correspondence} \eqref{eq:spanXSY} from $X\in\Aff_B$ to $Y\in\Sch_B$
is the data of
\begin{itemize}
\item[(1)]
a projective $X$-scheme $\ovS$ of relative dimension $d$,
and an ample bundle $\mathcal O(1)$ on $\ovS$ with a section $t_\infty\in \Gamma(\ovS, \mathcal O(1))$
such that the closed subscheme $S_\infty=Z(t_\infty)$ has positive relative codimension in $\ovS$ over $X$
\item[(2)]
an isomorphism
$S\cong(\ovS-S_\infty)$
\end{itemize}
We refer to
$(\ovS,\mathcal O(1),t_\infty)$
or
$(\ovS,S_\infty)$
as the \emph{compactified support}.
\end{definition}

\begin{definition}[Compactified framed correspondences]\label{def:compactifiedFrovS}
Suppose we have a $d$-dimensional affine correspondence \eqref{eq:spanXSY}
from $X\in\Aff_B$ to $Y\in\Sch_B$
equipped with
\begin{itemize}
\item[(1)] 
a framing $\Phi=(S, \varphi, g)$,
and 
\item[(2)] 
a compactification with compactified support
$(\ovS,S_\infty)$, see \Cref{def:compactifiedcorr}.
\end{itemize}
The data 
$\overline\Phi=(\ovS,S_\infty,\Phi)$
is called a \emph{compactified $d$-dimensional framed correspondence} from $X$ to $Y$.
We denote by $\cFr_n^d(X,Y)$ the set of $d$-dimensional level $n$ framed correspondences from $X$ to $Y$.
\end{definition}

\begin{remark}[Compactified framed correspondences]\label{rem:compactifiedFrovS}
\Cref{def:compactifiedcorr,def:compactifiedFrovS}
inform us that when $X\in\Aff_B$, $Y\in\Sch_B$,
an element in $\cFr_n^d(X,Y)$ is given by
\begin{itemize}
\item[(1)]
a $d$-dimensional framed correspondence 
$\Phi=(S, \varphi, g)$
from $X$ to $Y$, see \Cref{def:reldimFr}
\item[(2)] 
a projective scheme $\ovS$ over $B$ of relative dimension $d$,
and an ample bundle $\mathcal O(1)$ on $\ovS$ with a section $t_\infty\in \Gamma(\ovS, \mathcal O(1))$
such that the closed subscheme $S_\infty=Z(t_\infty)$ has positive relative codimension in $\ovS$ over $X$
\item[(3)]
an isomorphism
$S\cong(\ovS-S_\infty)$
\end{itemize}
\end{remark}

\begin{example}
\label{ex:cFrn0congFrn}
For any $c=(\overline{S},S_\infty,(S,\varphi,g))\in\cFr_n^0(X,Y)$, 
we have
$\dim_X \overline{S}=0$,
$S_\infty=\emptyset$,
and
$S\cong\overline{S}-S_\infty=\overline{S}$.
It follows that $S$ is a $0$-dimensional projective scheme over $X$,
and consequently, $S$ is finite over $X$.
Thus, 
we are entitled to the canonical morphism 
$$
\cFr_n^0(X,Y)\to \Fr_n(X,Y); c\mapsto c.
$$
\end{example}

\subsubsection{Extended correspondences}

For the definition of extended framed correspondences, 
we work over a local base scheme $B$ with closed point $z$ and generic point $\eta$. 

\begin{definition}[Extended compactified correspondences]\label{def:compactifiedfocusedcorr}
Suppose $X\in \Aff_B$, $Y\in \Sch_\eta$, and set $X_\eta=X\times_B \eta$.
An \emph{extended compactification} of 
a $d$-dimensional affine correspondence 
from $X_\eta$ to $Y$ over $\eta$ 
\begin{equation}\label{eq:def:compactifiedfocusedcorr:spanXetaSetaY}
X_\eta\leftarrow S_\eta\rightarrow Y
\end{equation}
is the data of
\begin{itemize}
\item[(1)]
a compactification of \eqref{eq:def:compactifiedfocusedcorr:spanXetaSetaY} over $\eta$, see \Cref{def:compactifiedcorr}, whose compactified support
is denoted by $(\ovS_\eta,\mathcal O(1)_\eta,t_{\infty,\eta})$
\item[(2)] 
a projective $X$-scheme $\ovS$ of relative dimension $d$,
and an ample bundle $\mathcal O(1)$ on $\ovS$ with a section $t_\infty\in \Gamma(\ovS, \mathcal O(1))$
such that for $S_\infty=Z(t_\infty)$, the closed subscheme $S_{\infty,\eta}=S_\infty\times_B\eta$ has positive relative codimension in $\ovS\times_B\eta$ over $X_\eta$
\item[(3)]
an isomorphism
$\ovS_\eta\cong\ovS\times_B \eta$
such that 
$(\mathcal O(1)_\eta,t_{\infty,\eta})$
is the base change of
$(\mathcal O(1),t_{\infty})$
\end{itemize}

We refer to $(\ovS,\mathcal O(1),t_{\infty})$ or $(\ovS,S_\infty)$
as the \emph{extended compactified support}.
\end{definition}

\begin{definition}[Extended compactified framed correspondence]\label{def:compactifiedfocusedFr}
Suppose that 
\begin{equation}\label{eq:spanXetaSetaY}
X_\eta\leftarrow S_\eta\rightarrow Y
\end{equation}
is a $d$-dimensional affine correspondence
from $X_\eta$ to $Y\in \Sch_\eta$ over $\eta$
equipped with 
\begin{itemize}
\item[(1)] 
a framing of \eqref{eq:spanXetaSetaY},
i.e., 
a $d$-dimensional framed correspondence
$\Phi=(S_\eta, \varphi, g)$ over $\eta$,
see \Cref{def:reldimFr},
and
\item[(2)] 
an extended compactification of \eqref{eq:spanXetaSetaY}, 
see \Cref{def:compactifiedfocusedcorr}, 
with extended compactified support $(\ovS,S_\infty)$
\end{itemize}
The data
$\overline\Phi=(\ovS,S_\infty,\Phi)$ is called an
\emph{extended compactified $d$-dimensional framed correspondence} from $X_\eta$ to $Y$.
We denote by $\cFr_n^d((X,X_\eta),Y)$ the set of extended $d$-dimensional level $n$ framed correspondences 
from $(X,X_\eta)$ to $Y$. 
\end{definition}

\begin{remark}[Extended compactified framed correspondence]\label{rem:compactifiedfocusedFr}
\Cref{def:compactifiedcorr,def:compactifiedFrovS,def:compactifiedfocusedcorr,def:compactifiedfocusedFr} 
informs us that when $X\in\Aff_B$, $Y\in\Sch_\eta$,
an element in $\cFr_n^d((X,X_\eta),Y)$ is given by
\begin{itemize}
\item[(1)] 
a $d$-dimensional framed correspondence 
$\Phi=(S_\eta, \varphi, g)$
from $X_\eta$ to $Y$ over $\eta$ (see also \Cref{def:reldimFr})
\item[(2)]
a projective scheme $\ovS$ over $X$ of relative dimension $d$,
and an ample bundle $\mathcal O(1)$ on $\ovS$ with a section $t_\infty\in \Gamma(\ovS, \mathcal O(1))$
such that for $S_\infty=Z(t_\infty)$, the closed subscheme $S_{\infty,\eta}=S_\infty\times_B\eta$ has positive relative codimension in $\ovS\times_B\eta$ over $X_\eta$
\item[(3)]
an isomorphism
$S_\eta\cong(\ovS-S_\infty)\times_B \eta$
\end{itemize}
\end{remark}

\begin{remark}
    By base change along $\eta\to B$,
    an extended compactified framed correspondence $c\in\cFr_n^d((X,X_\eta),Y)$ 
    defines a compactified framed correspondence $c_\eta\in\cFr_n^d(X_\eta,Y)$.
    We write ``extended" in the definitions above since $c$ extends $c_\eta$ 
    along the open immersion $X_\eta\to X$.
\end{remark}
\begin{example}\label{ex:defc}
\Cref{rem:compactifiedFrovS,rem:compactifiedfocusedFr} show that 
    every $d$-dimensional compactified framed correspondence from $X$ to $Y$ over $B$
    defines an extended $d$-dimensional compactified framed correspondence from 
    $X\times_B \eta$ to $Y\times_B \eta$.
\end{example}

\begin{definition}\label{def:smoothcontain}
Suppose $X\in \Aff_B$, $Y\in \Sch_\eta$, and $X_\eta=X\times_B \eta$.
Let $r\colon X_\eta\to Y$ be a morphism of schemes.
Given an extended compactified $d$-dimensional framed correspondence $(\ovS,S_\infty,\Phi)$ from $X_\eta$ to $Y$, 
we let $S$ denote the complement of $S_\infty$ in $\overline S$.
Furthermore, we say that $(\ovS,S_\infty,\Phi)$ \emph{smoothly contains} $r$ if
there is a closed subscheme $\Gamma$ in $S=\ovS-S_\infty$
such that the following hold:
\begin{itemize}
\item[(1)] the canonical morphism $S\to X$ induces an isomorphism $\gamma\colon \Gamma\cong X$
\item[(2)] the canonical morphism $S\to X$ is smooth at each point of $\Gamma$
\item[(3)] 
the morphism $r$ coincides with the composite 
\[
X_\eta\xrightarrow{\gamma^{-1}_\eta} \Gamma\times_B \eta\hookrightarrow 
S_\eta\to (\A^n_{X_\eta})^h_{S_\eta}\xrightarrow{g} Y
\]
Here $\gamma_\eta$ is the isomorphism $\gamma_\eta\colon \Gamma\times_B \eta\xrightarrow{\cong} X_\eta$, 
and $\overline\Phi=(\ovS,S_\infty,\Phi)$
\end{itemize}
We refer to $\Gamma$ as the \emph{compactified graph} of $r$ in $\ovS$.

\begin{example}
    Given a $d$-dimensional compactified framed correspondence $c$ from $X$ to $Y$ over $B$
    with support $S$ as in \Cref{def:compactifiedFrovS},
    along with a morphism $\overline{r}\colon X\to Y$ that passes through $S$,
    then the extended $d$-dimensional compactified framed correspondence from $X\times_B \eta$ to $Y$ from \Cref{ex:defc}
    smoothly contains the morphism $\overline{r}\times_B \eta$.
\end{example}

\end{definition}

\subsection{Construction of extended compactified framed correspondences}\label{sect:FocCompFr:Constr}

In the following lemma, we explain how an appropriate compactification of a scheme $X$ over a local base scheme $B$ defines an extended compactified framed correspondence from the generic point $\eta$ in $B$ to the generic fiber $X_\eta$.

\begin{lemma}\label{lm:fcFrofdimensiond=dimX}
Suppose the base scheme $B$ is local and irreducible with closed point $z\in B$ and generic point $\eta\in B$.
Let $X\in \Sm_B$, $x\in X\times_B z$, $d=\dim_B X_x$, $Z$ be a closed subscheme in $X_\eta=X\times_B \eta$, 
and set $X_{x,\eta}:=X_x\times_B \eta$.
Then, there exists an 
extended compactified $d$-dimensional framed correspondence 
$\overline \Phi=(\overline X,X_\infty,\Phi)$ from $(B,\eta)$ to $X_\eta$ such that the morphism $\overline X\to B$ 
is smooth at $x$, and $x\not\in X_\infty$. 
\end{lemma}

\begin{proof}
Without loss of generality, we can assume that $X$ is affine with a trivial tangent bundle, 
since there is such a Zariski neighborhood for all $x\in X$. 
Then there is a closed embedding $X\not\hookrightarrow \A^N_B$ with a trivial normal bundle.
Let $\overline X$ be the closure of $X$ in $\PP^N_B$. 
Then $\overline X\to B$ is smooth at $x$ because $X\in \Sm_B$, $x\in X$.
Let $t_\infty\in \Gamma(\PP^N_B,\mathcal O(1))$ be a section such that $Z(t_\infty)=\PP^N_B-\A^N_B$,
and denote by the same symbol the pullback of $t_\infty$ to $\overline X$. 
Set $X_\infty=Z(t_\infty)$ so that $X=\overline X-X_\infty$ and $x\not\in X_\infty$. 
Since the normal bundle $N_{X/\A^N_B}$ is trivial, 
there is a set of functions $(f_{d+1},\dots ,f_{N})\in \mathcal O(\A^N_B)$ such that $X$ is a 
clopen subscheme in $Z(f_{d+1},\dots,f_N)$ and a retraction $\mathrm{pr}\colon (\A^N_B)^h_X\to X$.
Then $(X,f_{d+1},\dots, f_N,\mathrm{pr})$ is a $d$-dimensional framed correspondence from $B$ to $X$.
Its image under the base change along $\eta\to B$ is the $d$-dimensional framed correspondence 
$(X_\eta,f_{d+1},\dots, f_N,\mathrm{pr})$ from $\eta$ to $X_\eta$.
In summary, we obtain the 
extended compactified framed correspondence
$(\overline X, X_\infty, (X_\eta,f_{d+1},\dots, f_N,\mathrm{pr}) )$ from $(B,\eta)$ to $X_\eta$.
\end{proof}

In the technical results 
\Cref{lm:ovcalXUk,lm:SectionsonCompactifiedobjecthorizotnalandverticleinductions,lm:EqDimoverDimOne,lm:onedimensionalfocusedcompactified},
we decrease the dimension of the correspondence provided by \Cref{lm:fcFrofdimensiond=dimX} down to $1$,
and obtain an element in $\cFr_n^{1}((X_{x,\eta},X_x), X_{\eta})$,
which smoothly contains the canonical morphism $X_x\to X$.

\begin{lemma}
\label{lm:ovcalXUk}
Let $U$ be a local scheme with closed point $\upsilon$.
Let $\ovcalX$ be a projective $U$ scheme with a very large line bundle $\calO(1)$.
Let $\calZ_j\not\hookrightarrow \ovcalX$ be a finite set of closed subschemes, $j\in J$, 
and suppose that $x\in\ovcalX_\upsilon$ is a $\upsilon$-rational point. 
Assume that the following conditions hold.
\begin{itemize}
\item[(s1)] $\ovcalX\to U$ is smooth at $x$
\item[(s2)] $\codim_{\ovcalX_\upsilon} {\calZ_j\times_U\upsilon}=c_j>0$ for all $j\in J$
\item[(s3)] $x\not\in \calZ_j$ for each $j$ such that $c_j>1$
\end{itemize}
Let $\{\mathcal L_r\}$ be a finite set of line bundles on $\overline\calX$ and $d=\dim_\upsilon \ovcalX_\upsilon>0$, 
where $\ovcalX_\upsilon=\ovcalX\times_U \upsilon$.
Then for $l_n\gg 0$, $n=2,\dots,d$,
there exist sections $s_{n}\in \Gamma(\ovcalX, \mathcal O(l_n))$ such that 
\begin{itemize}
\item[(t1)] $\ovcalC=Z(s_{2},\dots, s_{d})$ contains $x$ and is smooth at $x$ over $U$
\item[(t2)] $\codim_{\ovcalC\times_U\upsilon} ({\calZ_j\cap \ovcalC}\times_U\upsilon)=c_j$
\item[(t3)] the restricted line bundle $\mathcal L_r\big|_{\calZ\cap\ovcalC}$ is trivial for all $r$, 
where $\calZ=\bigcup_{j\in J} \calZ_j$
\end{itemize}
\end{lemma}

\begin{proof}
We prove the claim by induction on $d$.

Let $d=1$.
Then the set of sections $\{s_2,\dots,s_d\}$ is empty and the properties (t1) and (t2) hold by (s1) and (s2).
Since $\ovcalX$ is a projective scheme, $\dim_{\upsilon} \ovcalX\times_U\upsilon=1$, $U$ is local, 
and using \cite[Corollary A.3]{nonperfect-SHI}
we conclude that $\dim_U \ovcalX\leq 1$.
Since $Z$ is a closed subscheme such that $Z\times_U\upsilon$ is of positive codimension in $\ovcalX\times_U\upsilon$,
it follows that $Z$ is finite over $U$. 
Consequently, $Z$ is a semi-local affine scheme.
Therefore, each line bundle $\mathcal L_r\big|_{Z}$ is trivial, so (t3) is valid.

If $\dim_\upsilon \ovcalX_\upsilon=d>1$, suppose that the claim is proven for all projective schemes 
$\ovcalX$ over $U$ such that $\dim_\upsilon \ovcalX\times_U\upsilon<d$.
For $l_d\gg 0$, we will construct a section $s_d\in \Gamma(\ovcalX, \calO(l_d))$ such that
\begin{itemize}
\item[(t1')]
$s_d\big|_x=0$, and $Z(s_d)$ is smooth at $x$
\item[(t2')]
$\codim_{\ovcalX\times_U\upsilon} (Z(s_d\big|_{\ovcalX\times_U\upsilon}))=\codim_{Z_j\times_U\upsilon} (Z(s_d\big|_{Z_j\times_U\upsilon}))=1$ for all $j\in J$
\end{itemize}
Then the claim follows by the inductive assumption applied to $\ovcalX^\prime=Z(s_d)$.
To obtain the required section $s_d$, 
we consider the cotangent vector space $(\Omega_{\overline \calX_\upsilon})\big|_x=\Omega_{\overline \calX_\upsilon}\otimes_{\mathcal O(\ovcalX)} \mathcal O(x)$, i.e., the vector space of sections of 
the cotangent bundle of $\overline\calX_\upsilon$ at $x$.
Note that this vector space coincides with the conormal space of $x$ in $\overline \calX_\upsilon$.
Choose a non-trivial cotangent vector $dx\in (\Omega_{\ovcalX_\upsilon})\big|_x$.
Consider
the closed subscheme in $\ovcalX_\upsilon$ that is the vanishing locus
$Z(I^2(x))$ of the square of the ideal $I(x)$,
i.e., the first-order thickening of $x$ in $\ovcalX_\upsilon$.
Then $dx$ defines a regular function on $Z(I^2(x))$, which we also denote by $dx$. Furthermore, for any lifting $f\in \mathcal O((\overline X_\upsilon)_x)$ on the local scheme of $\overline X_\upsilon$ at $x$ such that $f\big|_{Z(I^2(x))}=dx$, the vanishing locus $Z(f)$ is essentially smooth over $\upsilon$.
The line bundle $\mathcal O(1)$ is trivial on the local scheme $(\overline X_\upsilon)_x$ and $Z(I^2(x))$. 
Due to the trivialization,
we obtain a section $dx\in \Gamma(Z(I^2(x)), \mathcal O(1))$.
By (c3), it follows that there exists a finite set of closed points $F=\{z_1,\dots, z_L\}\subset \ovcalX$, 
such that $F$ intersects non-emptily each irreducible component of $\ovcalX$ and $Z_j$ for each $j$, 
and $x\not\in F$.
Serre's theorem on ample bundles \cite[III, Corollary 10.7]{Ha77} 
ensures that for $l_d\gg 0$, 
there is a section $s_d$ such that 
\[
s_d\big|_{x}=0,\quad s_d\big|_{Z(I^2(x))}=dx,\quad s_d\big|_{z}\neq 0\text{ for all $z\in F$}
\]
The first two equalities above imply (t1'), and the latter condition implies (t2').
\end{proof}

\begin{lemma}
\label{lm:SectionsonCompactifiedobjecthorizotnalandverticleinductions}
Suppose $B$ is local with closed point $z\in B$ and let $\eta\in B$. 
Let $U$ be an essentially smooth local scheme over $B$, where $\upsilon\in U$ is the closed point.
Let $\ovX$ be a projective scheme over $B$, $\dim_\eta \ovX_\eta=d$, where $\ovX_\eta=\ovX\times_B \eta$. 
We set $\ovcalX=\ovX\times_B U$, $U_\eta=U\times_B \eta$, and $\ovcalX_\eta=\ovcalX\times_B \eta$.

Let $Z_j\subset \ovX_{\eta}$ be a finite set of closed subschemes indexed by $j\in J$, 
and let $\overline{Z_j}$ be the closure of $Z_j$ in $\ovX\times_B {B_{\eta^\prime}}$.
Let $x$ be a $\upsilon$-rational point of $\ovcalX\times_U \upsilon$. 
Suppose 
\begin{itemize}
\item[(s1)] $\ovcalX\to U$ is smooth at $x$
\item[(s2)] $\codim_{X_\eta\to\eta} {Z_j}=c_j>0$
\item[(s3)] $x\not\in \ovcalZ_j$ for each $j$ such that $c_j>1$, where $\ovcalZ_j= \overline{Z_j}\times_B U$
\end{itemize}
Let $\mathcal L_r$ be a finite set of line bundles on $\ovcalX$, and $\calO(1)$ be an ample bundle on $\overline\calX$.

Then 
for $l_j\gg 0\in \mathbb Z$, $j=2,\dots,d$,
there is a set of sections $s_{j}\in \Gamma(\ovcalX, \mathcal O(l_j))$
such that 
\begin{itemize}
\item[(t1)] $\ovcalC=Z(s_{d},\dots, s_{2})$ contains $x$ and is smooth at $x$ over $U$
\item[(t2)] the relative codimension $\codim_{\ovcalC_\eta\to U_\eta} ({\calZ_j\cap \ovcalC_\eta})$ equals $c_j$, where $\ovcalC_\eta=\ovcalC\times_B \eta$
\item[(t3)] the restriction of line bundles $\mathcal L_r\big|_{\calZ\cap\ovcalC_\eta}$ are trivial for all $r$, where $\calZ=\bigcup_j \calZ_j$
\end{itemize}
\end{lemma}
\begin{proof}

For any point $\theta\in B$, denote by $\overline{\theta}$ the closure of $\theta$ in $B$.
We prove the claim by induction on $c=\codim_{\overline{\eta}} z$. 
When $c=0$, then ${\eta}=z$, and the claim is proven by \Cref{lm:ovcalXUk}.

Suppose that the claim holds for all $\eta\in B$ and $\codim_{\overline{\eta}} z<c$.
Given $\eta\in B$ such that $\codim_{\overline{\eta}} z=c$,
let $\eta^\prime\in B$ be a point such that $\eta^\prime\in \overline{\eta}$ and $\codim_{\overline{\eta}}\eta^\prime=1$. 
Then $\codim_{B}\eta^\prime=\codim_{B} \eta+1$,
and
$\codim_{\overline{\eta^\prime}} z<c$.
Consider fibers $\ovcalX_{\eta^\prime}=\ovcalX\times_B \eta^\prime$ and
$\overline{Z_j}_{\eta^\prime}=\overline{Z_j}\times_B \eta^\prime$.
By 
\Cref{lm:EqDimoverDimOne} we have $\codim_{\ovX_{\eta^\prime}} \overline{Z_j}_{\eta^\prime}=c_j$.
The inductive hypothesis asserts that, for $l_n\gg 0$, $n=2,\dots ,d$, 
there exist appropriate sections $s^\prime_n\in \Gamma(\ovcalX_{\overline{\eta^\prime}},\calO(l_n))$,
where
$\ovcalX_{\overline{\eta^\prime}}=\ovcalX\times_B{\overline{\eta^\prime}}$.
To conclude, 
we choose sections $s_j\in \Gamma(\ovcalX,\calO(l_j))$ such that $s_j\big|_{\ovcalX_{\overline{\eta^\prime}}}=s_j^\prime$.
\end{proof}

\begin{lemma}\label{lm:EqDimoverDimOne}
Suppose that the base scheme $B$ is $1$-dimensional. 
Then, any smooth projective morphism $Y\to B$ is equidimensional.
\end{lemma}
\begin{proof}
Recall that $Y\to B$ is equidimensional if and only if 
the naturally induced morphism $Y\times_B{B_{(w)}}\to B_{(w)}$
is equidimensional for each point $w\in B$, 
or equivalently that $Y^\prime\to B^\prime$ is equidimensional
for each irreducible component $Y^\prime$ of $Y$, 
where $B^\prime$ is the closure of the image of $Y^\prime$ in $B$.   
Thus, we may assume that $B$ is local and that $B$ and $Y$ are irreducible.
If $Y\times_B\eta=\emptyset$, 
where $\eta\in B$ is the generic point, 
then 
$Y\to B$ passes through a proper closed subscheme of $B$, and consequently $Y\to B$ is non-smooth.
It follows that $Y\times_B\eta\neq\emptyset$. 
Since $Y\to B$ is projective and $B$ is $1$-dimensional, 
it is also equidimensional due to \cite[Corollary A.3]{nonperfect-SHI}.
\end{proof}

\begin{lemma}\label{lm:onedimensionalfocusedcompactified}
Suppose $B$ is local with a generic point $\eta\in B$.
Let $X\in \Sm_B$, $x\in X$, $X_\eta=X\times_B \eta$, and suppose that $Z\subset X_{\eta}$ is a closed subscheme.
Then for $X_{x,\eta}=X_x\times_X X_{\eta}$
there exists an extended compactified $1$-dimensional framed correspondence
$\overline{\Phi}=(\overline{C},C_\infty,\Phi)\in \cFr_n^{1}((X_{x,\eta},X_x), X_{\eta})$ such that 
 \begin{itemize}
\item[(c1)] $\overline{\Phi}$ smoothly contains the canonical morphism $\mathrm{can}\colon X_{x}\to X$
with the compactified graph $\Delta$ in $\overline{C}$, $\Delta\cap C_{\infty}=\emptyset$
\item[(c2)] 
$C_{\infty,\eta}\cap \overline{\mathcal Z}=\emptyset$
\item[(c3)]
the restrictions of line bundles 
$\mathcal O(1)$, and $\mathcal O(\Delta)$
are trivial
on ${C_\infty}$, and $\overline{\mathcal Z}\cup \Delta$,
where 
$\overline{\mathcal Z}$ is the closure in $\overline C$ of the subscheme
$\mathcal Z_\eta = C_\eta\times_{X_{\eta}} Z$,
$C_{\infty,\eta}=C_\infty\times_B \eta$ 
\end{itemize}
\end{lemma}

\begin{proof}
\Cref{lm:fcFrofdimensiond=dimX} shows that for $d=\dim^x_B X$ there is an extended compactified 
$d$-dimensional framed correspondence 
$\overline \Phi=(\overline X,X_\infty,\Phi)$ from $\eta$ to $X_\eta$ such that the morphism 
$\overline X\to V_\eta$ is smooth at $z$, 
and $z\not\in X_\infty$.
Base change along $X_x\to B$ gives an extended compactified $d$-dimensional framed correspondence 
\begin{equation}
\label{eq:ovSSinftySetavarphig}
(\overline S, S_\infty, (S_\eta,\varphi_{d+1},\dots, \varphi_n,g) )
\end{equation} 
from $(X_x,X_{x,\eta})$ to $X_\eta$.
Let $\Delta\not\hookrightarrow X_x\times_B X$ be the graph of the canonical map $X_x\to X$.
Since the morphism $X_x\times_B X\to X_x$ is smooth,
the framed correspondence \eqref{eq:ovSSinftySetavarphig}
smoothly contains the map $X_{x,\eta}\to X_\eta$, 
and $\Delta$ is its compactified graph in the sense of \Cref{def:smoothcontain}. 
Finally, the morphism $S=\overline S-S_\infty=X\times_B X_x\to X_x$ is smooth, because $X$ is smooth over $B$.
Therefore, 
we obtain a $d$-dimensional framed correspondence satisfying the properties (c1)-(c3); 
in addition, 
$\overline S=\overline X\times_B X_x$ for $\overline X\in \Sch_B$.

We set $\calZ^S_\eta = S\times_{X_\eta} Z$, $S_{\infty,\eta}=S_\infty\times \eta$
and apply \Cref{lm:SectionsonCompactifiedobjecthorizotnalandverticleinductions} to 
\begin{itemize}
\item $U=X_x$; here $x\in X_x$ plays the role of $\upsilon\in U$ 
\item $\overline X$ with $\calX = S$ and $\ovcalX = \overline S$
\item $x\in \calX$ equal to $\Delta\times_{X_x}x$ in $S$, 
and the closed subschemes $\calZ^{S}_\eta$ and $\calZ^{S}_\eta\cap S_{\infty,\eta}$
\end{itemize}
There exist sections $s_2,\dots, s_d$ on $\overline S$ satisfying (t1)-(t3) in  
\Cref{lm:SectionsonCompactifiedobjecthorizotnalandverticleinductions}.
For $C_\eta=C\times_B\eta$, we set
\[\begin{array}{lll}
\overline C&=& Z(s_2,\dots,s_{d})\subset \overline S \\
C_\infty &=& \overline C\cap S_\infty\\
\Phi &=& (C_\eta,s_2/t_\infty^{l_2},\dots,s_d/t_\infty^{l_d},\varphi_1,\dots,\varphi_n,g)
\end{array}\]
With these definitions, 
(t1)-(t3) in \Cref{lm:SectionsonCompactifiedobjecthorizotnalandverticleinductions}
imply the claims (c1)-(c3) of the lemma.
\end{proof}

In \Cref{th:MoveFropsubschgenfibsmoothloc},
we obtain a $0$-dimensional correspondence starting from 
the $1$-dimensional one provided by \Cref{lm:onedimensionalfocusedcompactified}.

\begin{theorem}
\label{th:MoveFropsubschgenfibsmoothloc}
Suppose $B$ is local with generic point $\eta$.
For $X\in \Sm_B$, $x\in X$, we set $X_{\eta}:=X\times_B \eta$, $X_{x,\eta}:=X_x\times_X X_{\eta}=X_x\times_B {\eta}$, 
and assume $Z\not\hookrightarrow X_{\eta}$ is a closed subscheme of positive codimension.
Then, for some $n>0$, 
there exist  
\[
c=(Z,\psi,g)\in \cFr^0_n({X_{x,\eta}}\times\A^1, X_{\eta}),
\]
and closed subschemes $\Delta$, $Z^+$, $Z^-$ in $Z$
satisfying
\begin{itemize}
\item[(1)] 
$Z=Z(\psi)$,
$Z\times_{\A^1}\{0\}= \Delta\amalg Z^-$,
$Z\times_{\A^1}\{0\}= Z^+$
\item[(2)] 
$Z^-\times_{g,X_\eta,i}Z=\emptyset$,
$Z^+\times_{g,X_\eta,i}Z=\emptyset$
\item[(3)]
the morphism $\mathrm{pr}\big|_{\Delta}\colon\Delta\to X_{x,\eta}$
induced by the projection $\mathrm{pr}\colon \A^n_{X_{x,\eta}}$
is an isomorphism
\item[(4)]
the morphism $g\big|_{\Delta}(\mathrm{pr}\big|_{\Delta})^{-1}$ equals
the canonical morphism $\mathrm{can}\colon X_{x,\eta}\to X_{\eta}$
\end{itemize}
Here 
$i_0,i_1\colon X_{x,\eta}\to X_{x,\eta}\times\A^1$ are the $0$- and $1$-sections.
\end{theorem}

\begin{proof}
A $1$-dimensional extended compactified framed correspondence $(\overline C,C_\infty,\Phi)$ that satisfies the conditions (c1)-(c3) exists by \Cref{lm:onedimensionalfocusedcompactified}.
Since $(\overline C,C_\infty,\Phi)$ 
smoothly contains the map $X_{x,\eta}\to X_\eta$ according to (c1), 
it follows that $\Delta$ is a divisor on $\overline C$.
Next we consider an element
\[\delta\in \Gamma(\overline C, \mathcal O(\Delta)), \; Z(\delta)=\Delta.\]
By (c1) and \Cref{def:smoothcontain} there is a closed scheme
$\Delta\not\hookrightarrow \overline C$ which is the compactified graph of $X_{x,\eta}\to X_\eta$.
So,the generic fiber $\Delta_\eta=\Delta\times_B \eta$ is the graph of $X_{x,\eta}\to X_\eta$,
and the latter map decomposes as $X_{x,\eta}\cong \Delta_\eta\to C_\eta\to X_\eta$.
Consequently, it follows that 
\[
\Delta_\eta\subset C_\eta=\overline C_\eta-C_{\infty,\eta},\quad
\Delta_\eta\cap C_{\infty,\eta}=\emptyset.\]
By (c2) $\overline{\mathcal Z}_\eta\cap C_{\infty,\eta}=\emptyset$, hence $\overline{\mathcal Z}_\eta\subset C_\eta$.
Note that since $\mathcal Z_\eta$ is a closed subscheme in $C_\eta$, it follows that \[\overline{\mathcal Z}_\eta=\mathcal Z_\eta.\]
By (c3), there are invertible sections
\[
\gamma_{\overline{\mathcal Z}\cup \Delta}\in \Gamma({\overline{\mathcal Z}}\cup \Delta,\mathcal O(\Delta)),\quad
\nu_{C_\infty}\in \Gamma(C_\infty,\mathcal O(1)).\\
\]
Recall that $C_\infty = Z(t_\infty)$, $t_\infty\in \Gamma(\overline C,\mathcal O(1))$.
For $l\gg 0$ there are sections $s_1$, $s_0^\prime$
such that
\begin{center}
\begin{tabular}{l r}
\toprule
$s_1\in \Gamma(\overline C_\eta,\calO(l) )$ &
$\tilde s_0\in \Gamma(\overline C_\eta,\calO(l-\Delta) )$  \\
\midrule
\vspace{6pt}$s_1|_{\widetilde\calZ_\eta} = t^l_\infty$ &
$\tilde s_0|_{\overline\calZ_\eta\cup\Delta}=t^l_\infty\gamma^{-1}_{\overline\calZ\cup\Delta}$\\
$s_1|_{C_{\infty,\eta}} = \nu^l_{C_\infty}$ & $\tilde s_0|_{C_{\infty,\eta}} = \nu^l_{C_\infty}\delta|^{-1}_{C_\infty}$\\
\bottomrule
\end{tabular}
\end{center}
where $\calO(l-\Delta)$ denotes $\calO(\Delta)^{\otimes-1}(l)=\calO(\Delta)^{\otimes-1}\otimes\calO(l)$.
All four sections in the second and third rows of the table are invertible because $C_{\infty,\eta}\cap \overline{\mathcal Z}=\emptyset$, and $C_{\infty,\eta}\cap \Delta=\emptyset$. Consequently, $t_\infty\big|_{\overline{\mathcal Z}}$, $\delta\big|_{C_{\infty,\eta}}$ and $t_\infty\big|_{\Delta}$ are invertible.

Define $s=(1-\lambda)\tilde s_0\delta+\lambda s_1\in \Gamma(\overline C_{\eta}\times\A^1,\mathcal O(l))$.
Then the data
\[
c = (Z,\psi, g) \in \cFr^0_n(X_{x,\eta}\times\A^1,X_\eta),\quad
Z=Z(s),\, \psi=(s/t_\infty^l, \varphi),
\]
and $Z^-=Z(\tilde s_0)$, $Z^+=Z(s_1)$
satisfy properties (1)-(4). 
\end{proof}

\begin{theorem}\label{cor:InjectivityFrgenfibsmoothloc}
Suppose that the base scheme $B$ is local. Let $\eta\in B$ be a point.
Let $X\in \Sm_B$, $X_\eta=X\times_B \eta$, 
and $X_{x,\eta}=X_x\times_X X_{\eta}$, where $x\in X$ is a point.
Let $Z$ be a closed subscheme of positive codimension in $X_\eta$.
Then there is a framed correspondence $c\in \ZF_n(X_{x,\eta}\times\A^1, X_{\eta})$ such that
\begin{itemize}
\item[(1)] $c\circ i_0=\sigma^n\mathrm{can}$, 
where $\mathrm{can}\colon X_{x,\eta}\to X_{\eta}$ is the canonical morphism
\item[(2)] $c\circ i_1=j\circ \tilde c$ for some $\tilde c\in\ZF_n(X_{x,\eta}, X_{\eta}-Z)$, 
where $j\colon X_{\eta}-Z\to X_{\eta}$ is the canonical open immersion
\end{itemize}
Here, 
$i_j\colon X_{x,\eta}\to X_{x,\eta}\times\A^1$ denotes the $j$-section, $j=0,1$.
\end{theorem}

\begin{proof}
The claim follows for $X$ over $B$ provided that
it holds for $X\times_B{\overline{\eta}}$ over the closure $\overline{\eta}$ of $\eta$.
Thus, we may assume $B$ is irreducible and $\eta$ is the generic point of $B$.
It suffices to show the equality 
\begin{equation}
\label{eq:sigmancaneqjtildecprime} 
[\sigma^n\mathrm{can}]=[j\circ\tilde c^{\prime}]\in \overline{\ZF}_n(X_{x,\eta},X_\eta)
\end{equation}
for some $\tilde c^{\prime}\in\ZF_n(X_{x,\eta},X_\eta-Z)$.
Let 
\[c^\prime=(Z,\psi, g) \in \Fr_n(X_{x,\eta}\times\A^1,X_\eta)\]
denote the framed correspondence furnished by \Cref{th:MoveFropsubschgenfibsmoothloc} 
through \Cref{ex:cFrn0congFrn}.
Then the closed subscheme $\Delta$ in $\A^n_{X_{x,\eta}}$ from \Cref{th:MoveFropsubschgenfibsmoothloc}
is a clopen subscheme of $Z$ by \Cref{th:MoveFropsubschgenfibsmoothloc}(1).
Therefore, the differential of $\psi$ at $\Delta$ defines an invertible matrix 
$\tau\in\mathrm{GL}_n(\Delta)\cong\mathrm{GL}_n(X_{x,\eta})$. 
We note that the isomorphism $\mathrm{GL}_n(\Delta)\cong\mathrm{GL}_n(X_{x,\eta})$
is induced by the isomorphism
\begin{equation}
\label{eq:prDelta}
\mathrm{pr}\big|_{\Delta}\colon\Delta\to X_{x,\eta}
\end{equation} 
provided by \Cref{th:MoveFropsubschgenfibsmoothloc}(3). 
The element 
    \[\tau^{-1}(\psi)\in\mathcal O_{\A^n_{X_{x,\eta}}}(\A^n_{X_{x,\eta}})^{\oplus n}\]
    comprised of regular functions gives rise to framed correspondences
    \[c^{\prime\prime}:=(Z,\tau^{-1}(\psi), g) \in \Fr_n(X_{x,\eta}\times\A^1,X_\eta) \]
    and 
    \[
    \tilde c^{\prime,-}:=(Z^-, i_0^*(\tau^{-1}(\psi)), i_0^*(g)),
    \tilde c^{\prime,+}:=(Z^+, i_1^*(\tau^{-1}(\psi)), i_1^*(g))
    \in  \Fr_n(X_{x,\eta}\times\A^1,X_\eta-Z) \]
Using these definitions, 
we obtain equalities in $\ZF_n(X_{x,\eta},X_\eta)$, 
\[
c^{\prime\prime}\circ i_0=d+j\circ\tilde c^{\prime,-},
\quad 
c^{\prime\prime}\circ i_1=j\circ\tilde c^{\prime,1}, 
\]
where
    \[d=(\Delta,\tau^{-1}(\psi), g))\in\Fr_n(X_{x,\eta},X_\eta).\]
By setting $\tilde c^{\prime}:=-\tilde c^{\prime,-}+\tilde c^{\prime,+}$ the above implies 
\[
[d]=[j\circ\tilde c^{\prime}]\in \overline{\ZF}_n(X_{x,\eta},X_\eta).
\]
Due to the isomorphism \eqref{eq:prDelta} and since the differential of $\tau^{-1}(\psi)$ 
at the closed subscheme $\Delta$ in $\A^n_{X_{x,\eta}}$ equals the identity matrix in 
$\mathrm{GL}_{n}(X_{x,\eta})$, 
it follows that
    \[[d]=[\sigma^n\mathrm{can}]\in \overline{\ZF}_n(X_{x,\eta},X_\eta).\]
    This implies the equality \eqref{eq:sigmancaneqjtildecprime}.
\end{proof}

\section{The main results on motivic Cousin complexes}\label{sect:CousComplexProof}

In this section, we prove that the canonical morphism \eqref{eq:CoutftoCou} between the Cousin complex and $\tf$-Cousin complex is a quasi-isomorphism and investigate the acyclicity of the Cousin complex above $\dim B$;
see \Cref{th:CoutfsimeqCou,thm:trivcohhdimBCous} and \Cref{cor:CoutfsimeqCou}.
We note that, in general, the Cousin complex over essentially smooth local schemes is \emph{not} acyclic; see \Cref{section:counterex} for a counterexample.
However, we prove the acyclicity of the Cousin complexes of fibers over points in the base scheme and the acyclicity of the columns of the Cousin bicomplex of essentially smooth local schemes; see \Cref{th:FnisExactoffiber,th:LnisFExactonfiber}.

\subsection{Cousin complexes on the fibers}
In this section, we aim to prove the acyclicity of the Cousin complex of fibers over points of the base scheme; see \Cref{th:FnisExactoffiber} below, and deduce the acyclicity of the columns of the Cousin bicomplex, see \Cref{th:LnisFExactonfiber}.
\begin{theorem}\label{th:FnisExactoffiber}
For 
any quasi-stable framed $\calF\in\SHsAN(z)$,
$z\in B$, 
and any essentially smooth local scheme $U\in \EssSm_B$, 
for each $l\in\bbZ$,
the morphism \eqref{eq:FtoCou} induces a canonical quasi-isomorphism
\begin{equation}
\label{eq:F(Uz)Cou(Uz,F)}
\pi_l\calF({U_z})\xrightarrow{\cong} \pi_l\Cou({U_z},\calF)
\end{equation}
in $\Kom(\Ab)$.
\end{theorem}

\begin{proof}
Let $U=X_{(x)}$ for $X\in \Sm_B$.
For a closed subscheme $Y\not\hookrightarrow X_z$, consider the spectra
$\calF_{Y}(X_z)$ and $\calF_{Y\times_{X_z}U}(U_z)$, and for immersions $Y\to Y^\prime$ of closed subschemes the morphisms
\begin{equation}\label{eq:extsupYYprime}\calF_{Y}(X_z)\to \calF_{Y^\prime\times_{X_z}U}(U_z).\end{equation}
We will show that the morphism
\begin{equation}\label{eq:extsupcodimccodimc-1}
\varinjlim_{\codim Y=c}\calF_{Y}(X_z)\to \varinjlim_{\codim Y^\prime=c-1}\calF_{Y^\prime\times_{X_z} U}(U_z)
\end{equation}
induces trivial
maps on homotopy groups.
Given $Y$ with $\codim Y=c>0$, then by \Cref{cor:InjectivityFrgenfibsmoothloc} there is a framed correspondence 
$r\in \Fr_n(U_z\times\A^1, X_z)$ such that 
\[
r_0= \sigma^n\mathrm{can},\quad r_1=e\circ r^\prime,
\]
where 
\[
\mathrm{can}\in \Fr_0(U_z, X_z), r^\prime\in \Fr_n(U_z, X_z-Y).
\]
We now aim to define a closed subscheme $Y^\prime\not\hookrightarrow X_z$ as the ``preimage'' of $Y\times_{X_z} U_z$ 
along the homotopy $r$.
To this end, write the framed correspondence $r$ as a span
\[
U_z\times\A^1\xleftarrow{f} Z\xrightarrow{g} X_z.
\] 
Let $W^\prime\subseteq U_z$ be the Zariski closure of the image of 
\[
f(g^{-1}(Y\times_{X_z} U_z))\subset U_z\times\A^1
\] 
along the projection $U_z\times\A^1\to U_z$. 
We define $Y^\prime$ as the minimal closed subscheme in $X_z$ such that $Y^\prime\times_{X_z} U_z=W^\prime$. 

Now, using that the morphism $f$ is finite, we have 
\[
\dim Y^\prime=\dim W^\prime =
\dim f(g^{-1}(Y\times_{X_z} U_z))= 
\dim g^{-1}(Y\times_{X_z} U_z),
\]
and hence
\begin{align*}
\codim_{X_z} Y^\prime &=
(\dim X_z-\dim(U_z\times\A^1)) + \codim_{U_z\times\A^1} f(g^{-1}(Y\times_{X_z} U_z))\\ 
&=-1 + \codim_{U_z\times\A^1} f(g^{-1}(Y\times_{X_z} U_z))\\
&=-1 + \codim_{Z} g^{-1}(Y\times_{X_z} U_z)\\
&=-1 + \codim_{U_z}(Y\times_{X_z} U_z)\leq \codim_{X_z} Y - 1.
\end{align*}
To show that \eqref{eq:extsupcodimccodimc-1} is trivial, it is enough to show that, 
for each $Y$ as above, 
the morphism \eqref{eq:extsupYYprime} is trivial.
Since $r$ induces a morphism
\[
\calF_{Y}(X_z)\xrightarrow{r^*} \calF_{(Y^\prime\times_{X_z}U)\times\A^1}(U_z\times\A^1),
\]
by $\A^1$-homotopy invariance of $\calF$ on $\Sm_z$ (see \Cref{citedth:SHI(k)}) we see that the morphism \eqref{eq:extsupYYprime} is equivalent to 
the morphism
\begin{equation}\label{eq:invimager1}
\calF_{Y}(X_z)\xrightarrow{r_1^*} \calF_{Y^\prime\times_{X_z}U}(U_z)
\end{equation}
induced by $r_1$.
Since $r_1=j\circ r^\prime$, the morphism \eqref{eq:invimager1} coincides with the composite 
\begin{equation*}
\calF_{Y}(X_z) \xrightarrow{j^*} \calF_{Y}(X_z-Y)\xrightarrow{(r^\prime)^*} \calF_{Y^\prime\times_{X_z}U}(U_z)
\end{equation*}
which is trivial since the first morphism is trivial.
Thus, \eqref{eq:extsupcodimccodimc-1} induces
trivial morphisms on homotopy groups,
and consequently the morphisms 
\begin{equation}\label{eq:extsupcodimccodimc-1U}
\varinjlim_{\codim Y=c+1, Y\subset U_z}\calF_{Y}(U_z)\to 
\varinjlim_{\codim Y^\prime=c, Y^\prime\subset U_z}\calF_{Y^\prime}(U_z)\end{equation}
for all $c=0,\dots ,\dim U_z-1$,
are trivial on homotopy groups.

Now we see that the differentials of the complex $\Cou({U_z},\calF)$ 
equal the composite morphisms
\[{\small \xymatrix{
\underset{y\in U_z^{(c-1)}}{\bigoplus}\calF_{y}(U_z)\ar[rr]\ar[rd]^{[1]} &&
\underset{y\in U_z^{(c)}}{\bigoplus}\calF_{y}(U_z)\ar[rr]\ar[rd]^{[1]} &&
\underset{y\in U_z^{(c+1)}}{\bigoplus}\calF_{y}(U_z)
\\
&
\underset{\codim_{U_z} Y^\prime=c}{\varinjlim}\calF_{Y^\prime}(U_z)\ar[ru] &&
\underset{\codim_{U_z} Y=c+1}{\varinjlim}\calF_{Y}(U_z)\ar[ru]\ar[ll]
&
}}\]

The horizontal arrow in the second row equals \eqref{eq:extsupcodimccodimc-1U} and induces trivial maps on homotopy groups.
The long exact sequences of homotopy groups associated with the triangles in $\SHtop$ 
\[
\underset{\codim_{U_z} Y=c+1}{\varinjlim}\calF_{Y}(U_z)
\to \underset{y\in U_z^{(c)}}{\bigoplus}\calF_{y}(U_z) \to 
\underset{\codim_{U_z} Y^\prime=c}{\varinjlim}\calF_{Y^\prime}(U_z)[1],\]
for $\codim_{U_z} Y=c+1$, $\codim_{U_z} Y^\prime= c$, split into short exact sequences in $\Ab$
\[
\underset{\codim_{U_z} Y=c+1}{\varinjlim}\pi_l\calF_{Y}(U_z)
\to \underset{y\in U_z^{(c)}}{\bigoplus}\pi_l\calF_{y}(U_z) \to 
\underset{\codim_{U_z} Y^\prime=c}{\varinjlim}\pi_l\calF_{Y^\prime}(U_z)[1].\]
Hence the sequence in $\Ab$
\[
\dots\to
\underset{y\in U_z^{(c+1)}}{\bigoplus}\pi_l\calF_{y}(U_z)\to 
\underset{y\in U_z^{(c)}}{\bigoplus}\pi_l\calF_{y}(U_z)\to
\underset{y\in U_z^{(c-1)}}{\bigoplus}\pi_l\calF_{y}(U_z)\to
\cdots
\]
is a long exact sequence.
Thus,
the morphism of complexes of abelian groups
\[
\pi_l\calF({U_z})\to \pi_l\Cou({U_z},\calF)
\]
is a quasi-isomorphism.
\end{proof}

\begin{lemma}\label{lm:InjectivitSpectrapresheaves}
Let $\calF\in\SHsA(z)$ be quasi-stable framed, $z\in B$,
and $U\in \EssSm_B$ be an essentially smooth local scheme.
If $\calF(\eta)\simeq 0$, 
where $\eta\in U_z$ is the generic point, 
then $\calF({U_z})\simeq 0$.
\end{lemma}

\begin{proof}
By \Cref{cor:InjectivityFrgenfibsmoothloc}, 
there exists a framed correspondence $r\in \Fr_n(U_z\times\A^1, X_z)$ such that 
\[
r_0= \sigma^n\mathrm{can}, r_1=e\circ r^\prime.
\]
Here 
\[
\mathrm{can}\in \Fr_0(U_z, X_z), r^\prime\in \Fr_n(U_z, X_z-Y)
\]
is given by the canonical morphism of schemes $U_z\to X_z$.
Since, 
for all $i\in \mathbb Z$,
the presheaf $\pi_i(\calF)$ is a framed $\A^1$-invariant presheaf because the same holds for $\calF$, 
the map of abelian groups
\[
\pi_i \calF(U_z)\to \pi_i \calF(\eta)
\]
is injective.
Since $\pi_i\calF(\eta)=0$ by assumption, the group $\pi_i\calF(U_z)$ vanishes.
\end{proof}

\begin{lemma}\label{lm:F(Uz)Fnis(Uz)}
For a quasi-stable framed $\calF\in\SHsA(z)$, $z\in B$, 
and any essentially smooth local scheme $U\in \EssSm_B$,
there is a canonical equivalence
\begin{equation}\label{eq:F(Uz)Fnis(Uz)}
\calF({U_z})\xrightarrow{\simeq} L_\nis\calF({U_z}).
\end{equation}
\end{lemma}
\begin{proof}
The $\A^1$-invariant presheaf $C=\cofib(\calF\to L_\nis\calF)$ satisfies the assumptions in 
\Cref{lm:InjectivitSpectrapresheaves}, 
because $C(U_z)\simeq0$, and because of \Cref{citedth:SHI(k)}.
\end{proof}

\begin{remark}\label{rem:lm:F(Uz)Fnis(Uz)}
\Cref{lm:F(Uz)Fnis(Uz)} equivalently claims
the equivalence
\begin{equation*}\label{eq:rem:calFXhczsimeqLniscalFXhxz}\calF((X^h_x)_z)\to L_\nis \calF((X_{(x)})_z).\end{equation*}
for any $X\in \Sm_B$, $z\in B$, $\calF\in \SH^{\fr,s}_{\A^1}(z)$, or equivalently,
the isomorphisms
\[
H^n_{\nis}((X^h_x)_z,\calF)\cong 
\begin{cases} \calF_\nis((X_{(x)})_z), & n=0\\
0,& n>0.
\end{cases}\]
for any $X\in \Sm_B$, $z\in B$, and $\calF\in \Pre^{\Ab}_{\A^1}(\Corr^\fr(z))$, where $\Pre^{\Ab}_{\A^1}(\Corr^\fr(z))$ is the $\infty$-category of $\A^1$-invariant presheaves of abelian groups on $\Corr^\fr(z)$.
\end{remark}

\begin{theorem}
\label{th:LnisFExactonfiber}
For a quasi-stable framed $\calF\in\SHsA(z)$, $z\in B$, 
and any essentially smooth local scheme $U\in \EssSm_B$,
for each $l\in\bbZ$,
there is a canonical quasi-isomorphism
\begin{equation}
\pi_l\calF({U_z})
\xrightarrow{\simeq}
\pi_l\Cou({U_z},L_\nis\calF).
\end{equation}
\end{theorem}

\begin{proof}
By
strict homotopy invariance over $z$ (\Cref{citedth:SHI(k)}), 
the sheaf $L_\nis\calF$ is $\A^1$-invariant.
Therefore, 
\Cref{th:FnisExactoffiber} yields an quasi-isomorphism $\pi_l L_\nis\calF(U_z)\simeq \pi_l\Cou(U_z,L_\nis\calF)$, 
and our claim follows by \Cref{lm:F(Uz)Fnis(Uz)}.
\end{proof}

\subsection{Cousin complexes with support}
Suppose that $i\colon Z\not\hookrightarrow B$ is a closed immersion.
We let $\Sm_{B,Z}$ (resp.~$\SmAff_{B,Z}$)
denote the category whose objects are schemes of the form $X^h_Z$ for $X\in\Sm_B$ (resp.~$X\in\SmAff_B$), 
see \cite[§4]{DKO:SHISpecZ},
Moreover, 
see \cite[§5]{DKO:SHISpecZ} for the category $\Smat_{B,Z}$ spanned by schemes with trivial tangent bundles 
in $\SmAff_{B,Z}$.
We have the following adjunctions:
\begin{align}
\label{eq:BZ-adj1}
i^*_{B,Z}\colon \SHsA(\mathcal{S}_{B,Z})&\rightleftarrows \SHsA(\mathcal{S}_B):i^{B,Z}_*\\
\tilde i_* \colon \SHsA(\mathcal{S}_{B,Z})&\rightleftarrows \SHsA(\mathcal{S}_B):\tilde i^! 
\label{eq:BZ-adj2}
\end{align}
that restricts to the subcategories of quasi-stable framed objects.
Here $\mathcal{S}_{B,Z}$ is $\Sm_{B,Z}$, $\SmAff_{B,Z}$, or $\Smat_{B,Z}$, 
and similarly for $\mathcal{S}_{B}$.
We refer to \cite[§6,§9]{DKO:SHISpecZ} for more details on these adjunctions.

\begin{lemma}\label{lm:arrowiustarpreserveA1invariance}
The functor $i_{B,z}^*\colon \SHs(\Sm^\mathrm{cci}_{B,z})\to \SHs(\Sm^\mathrm{cci}_{z})$ preserves $\A^1$-invariant quasi-stable framed objects.
\end{lemma}

\begin{proof}
The claim follows by \cite[Lemma 6.9]{DKO:SHISpecZ}, because $\Sm^\mathrm{cci}_{z}$
is a subcategory of $\Sm^\mathrm{cci}_{B*z}$ and the restriction functor preserves $\A^1$-invariant objects.
\end{proof}

\begin{lemma} 
\label{lm:ovis}
The restriction of the functor $i^*_{B,z}\colon \SHsA(\Smat_{B,z})\to \SHsA(\Smat_z)$ 
to the subcategories of 
quasi-stable framed objects
is an equivalence and
commutes with Nisnevich localization.
\end{lemma}

\begin{proof}
Without loss of generality, we may assume that $B$ is affine since $\Smat_{B,z}\simeq \Smat_{B^\prime,z}$.
Here $B^\prime$ is any Zariski neighborhood of $z$.
Note
that by \cite[Proposition 5.8]{DKO:SHISpecZ}
we have
\[
\SHsA(\Smat_{B,z})\simeq \SHsA(\SmAff_{B,z}),
\quad 
\SHsA(\Smat_z)\simeq \SHsA(\SmAff_z),
\]
and similarly for the subcategories of quasi-stable framed objects.
We refer to
\cite[Lemma 6.12]{DKO:SHISpecZ}
for the equivalence of 
the subcategories of quasi-stable framed objects induced by the functor 
\[
\SHsA(\Smat_{B,z})\to \SHsA(\Smat_z).
\]
By \cite[Lemmas 10.5, 10.8, 10.9]{DKO:SHISpecZ} $i^*_{B,z}$ preserves and detects Nisnevich local equivalences.
Therefore, $i^*_{B,z}$ commutes with $L_\nis$.
\end{proof}

\begin{proposition}\label{pr:CousinSmzB}
For any quasi-stable framed 
$\calF\in \SHsA(\Sm_{B,z})$ 
and essentially smooth local scheme $U\in \EssSm_B$,
for each $l\in\bbZ$, 
there is a canonical quasi-isomorphism
\[
\pi_l\calF(U^h_z)\simeq \pi_l\Cou(U^h_z,L_\nis\calF).
\]
\end{proposition}

\begin{proof}
Without loss of generality, 
we may assume that $z\in B$ is a closed point, 
$U\in \Sm^\mathrm{cci}_B$, 
and restrict $\calF$ to $\Sm^\mathrm{cci}_B$. 
To wit, 
$\Sm^\mathrm{cci}_{B,z}\simeq \Sm^\mathrm{cci}_{B_{(z)},z}$, 
where $B_{(z)}$ is the local scheme of $B$ at $z$, 
and by appeal to \Cref{lm:ovis}
the functors $i^*_{B,z}$ and ${i}_*^{B,z}$ yield equivalences
\begin{align*}
\SHfrsAnis(\Sm^\mathrm{cci}_{B,z}) &\simeq \SHfrsAnis(\Sm^\mathrm{cci}_{z}),\\
L_\nis\calF(U^h_z)&\simeq L_\nis i_{B,z}^*\calF(U_z),\\ \Cou(U^h_z,L_\nis\calF)&\simeq 
\Cou(U_z,L_\nis i^*_{B,z}\calF),
\end{align*}
and the subcategories of quasi-stable framed objects in $\SHsAN(-)$.
By applying \Cref{th:LnisFExactonfiber} to the $\A^1$-invariant quasi-stable framed object 
$i^*_{B,z}\calF\in \SHsA(\Sm^\mathrm{cci}_z)$, 
see \Cref{lm:arrowiustarpreserveA1invariance}, 
we get the equivalence $i_{B,z}^*\calF(U_z)\simeq \Cou(U_z,L_\nis i_{B,z}^*\calF)$.
\end{proof}

\begin{lemma}\label{lm:tis}
For any point $z\in B$,
the restriction of the functor $\tilde{i}^!j^*\colon \SHs_\tf(\Smat_{B})\to \SHs_\tf(\Smat_{B,z})$
to the subcategories of quasi-stable framed objects
preserves Nisnevich local and $\A^1$-invariant objects.
\end{lemma}

\begin{proof}
This follows from \cite[Propositions 9.7-9.9]{DKO:SHISpecZ}.
\end{proof}

\begin{lemma}\label{lm:tisFC}
Let $j\colon B_{(z)}\to B$ be the canonical morphism from the local scheme of $B$ at $z\in B$, 
and let $i\colon z\not\hookrightarrow B_{(z)}$ be the corresponding closed immersion.
For any essentially smooth local scheme $U\in \EssSm_B$ and quasi-stable framed $\calF\in \SHs(\Sm_{B})$, 
there is a canonical equivalence
\[
\tilde i^!j^*\calF(U^h_z)\simeq \calF_z(X).
\]
Moreover,
for any quasi-stable framed $\calF\in \SHs_\nis(\Sm_{B})$,
for each $l\in\bbZ$, there is a canonical quasi-isomorphism
\[
\pi_l\Cou(U^h_z,\tilde i^!j^*\calF)\simeq \pi_l\Cou_z(U,\calF).
\]
\end{lemma}

\begin{proof}
The first equivalence holds since, by the definition of $\tilde i^!j^*\calF(U^h_z)$, there are equivalences
\[
\fib(j^*\calF(U^h_z)\to j^*\calF(U^h_z\times_{B_{(z)}}(B_{(z)}-z)))\simeq 
\fib(\calF(U^h_z)\to \calF(U^h_z\times_B(B-z)))\simeq 
\calF_z(X).
\]
The second claim follows from the first claim.
\end{proof}

\begin{lemma}\label{lm:preiushrikcommuteLnis}
For any $z\in B$, the following holds:
\begin{itemize}
\item[(a)]
For any essentially smooth local henselian $U$ over $B$, and any quasi-stable framed $\calF\in \SHsAtf(\Sm_{B})$ 
such that $L_\nis \calF=0$, 
there is an equivalence
$\calF(U\times_B (B_{(z)}-z))\simeq0$.
\item[(b)]
For any essentially smooth local henselian $U$ over $B$, and any quasi-stable framed $\calF\in \SHsAtf(\Sm_{B})$ such that $L_\nis \calF=0$, there is an equivalence $\calF_z(U)\simeq0$.
\end{itemize}
\end{lemma}

\begin{proof}
We prove that (a) implies (b).
Let $U$ be an essentially smooth local henselian scheme.
Then $\calF(U)\simeq 0$, 
since $L_\nis \calF=0$ by the assumption in (b).
Therefore, $\calF(U\times_B (B-z))\simeq 0$ by (a).
It follows that $\calF_z(U)=\fib(\calF(U)\to \calF(U\times_B (B-z))\simeq 0$.

Part (a) holds when $\dim B=0$.
We proceed by induction, assuming that (a) holds for all $B$ of dimension $<d$. 
Then, 
as shown above, 
(b) holds for all such $B$. 
Let us show that (a) holds for any scheme $B$ of dimension $d$.
We set $V=(B_{(z)}-z)$, so that $\dim V=d-1$.
Denote by $\overline{V^{(i)}}$ the minimal closed subscheme of $V$ that contains $V^{(i)}$.
Since $\calF$ is $\tf$-local, 
the finite length filtration on the $S^1$-spectrum $\calF(U\times_B  V)$
given by the spectra $\calF(U\times_B  (V-\overline{V^{(i)}}))$,
leads to a convergent spectral sequence
\begin{equation}
\label{eq:F(fib_V)specseq}\bigoplus_{z\in V^{(i)}} 
\pi_{l+i}\calF_z(U) 
\Rightarrow 
\pi_{l}\calF(U\times_B  V).
\end{equation}
Applying (b) to $V$, 
viewed as a base scheme, 
it follows that the left side of \eqref{eq:F(fib_V)specseq} is trivial,
and hence $\calF(U\times_B  V)\simeq 0$.
\end{proof}

\begin{lemma}\label{lm:iushrikcommuteLnis}
For any $z\in B$, 
$\tilde i^!j^*$ commutes with the endofunctor $L_\nis$ on the subcategory of quasi-stable framed objects in
$\SHsAtf(\Sm_{B})$.
\end{lemma}
\begin{proof}
If $B^\prime\to B$ is a map of base schemes, 
the base change functor $\Pre(\Sm_B)\to \Pre(\Sm_{B^\prime})$ commutes with the Nisnevich localization. 
By \cite[Proposition 3.2.14]{five-authors}, the same holds for framed presheaves.
Consequently, 
the claim holds for $\SHsAtf(\Sm_{B})\to \SHsAtf(\Sm_{B^\prime})$, and thus for $j^*$.

Since the functor $\Sm_B\to \Sm_{B,z}$ given by $X\mapsto X^h_z$ 
preserves Nisnevich squares, it follows that
$\tilde i^!$ preserves Nisnevich local objects.
To conclude that $\tilde i^!$ commutes with $L_\nis$, 
it remains to prove that $\tilde i^!$ preserves Nisnevich equivalences, 
or equivalently, Nisnevich acyclic objects.
Suppose $L_\nis(\calF)\simeq 0$ for some quasi-stable framed $\calF\in \SHsAtf(\Sm_{B})$.
Then $\calF(U)\simeq 0$ for any essentially smooth local henselian $U$ over $B$.
\Cref{lm:preiushrikcommuteLnis} shows that $\calF(U-U_z)\simeq 0$.
Thus we have $\tilde i^!\calF(U)=\fib(\calF(U)\to \calF(U-U_z))\simeq 0$, 
and $L_\nis(\tilde i^!\calF)\simeq 0$.
\end{proof}

\begin{proposition}\label{pr:LnispreComumnsExact}
For any quasi-stable framed $\calF\in \SHsAtf(\Sm_{B})$, 
$z\in B$,
and essentially smooth local scheme $U\in \EssSm_B$,
for each $l\in\bbZ$,
the morphism \eqref{eq:FtoCou} induces a canonical quasi-isomorphism
\[
\pi_l\calF_z(U)\simeq 
\pi_l\Cou_z(U,L_\nis\calF)
\]
in $\Kom(\Ab)$.
\end{proposition}

\begin{proof}
Owing to \Cref{lm:tisFC,lm:iushrikcommuteLnis} there are equivalences
\[
\calF_z(U)
\simeq
\tilde i^!j^*\calF(U^h_z), 
\,\,\,
\Cou_z(U,L_\nis\calF)
\simeq 
\Cou(U^h_z,L_\nis\tilde i^!j^*\calF).
\]
\Cref{lm:tis} shows that $\tilde i^!j^*\calF$ is Nisnevich local and $\A^1$-invariant.
Thus, 
\Cref{pr:CousinSmzB} yields the quasi-isomorphism
\[
\pi_l\tilde i^!j^*\calF(U^h_z)\simeq \pi_l\Cou(U^h_z,L_\nis\tilde i^!j^*\calF).
\]
\end{proof}

\subsection{Cousin complexes, and $\tf$-Cousin complexes}
\begin{theorem}\label{th:CoutfsimeqCou}
For any
quasi-stable framed $\calF\in \SHsAtf(B)$, and
essentially smooth local scheme $U\in \EssSm_B$, 
for each $l\in\bbZ$,
the canonical morphism \eqref{eq:CoutftoCou}, see \Cref{def:CoutftoCoubi_AND_CoutftoCou},
induces a quasi-isomorphism
\[
\pi_l\Cou_\tf(U,L_\tf\calF)\simeq \pi_l\Cou(U,L_\nis\calF)
\]
in $\Kom(\Ab)$.
\end{theorem}
\begin{proof}
By \Cref{lm:bicompexsimeqlcomlex}, 
we reduce the claim to showing that the canonical morphism of bicomplexes \eqref{eq:CoutftoCoubi}
\[\pi_l\Cou_\tf(U,L_\tf\calF)\to \pi_l\Coubi(U,L_\nis\calF)\]
is a quasi-isomorphism.
By \Cref{pr:LnispreComumnsExact}
the morphism of columns \eqref{eq:FzCouzF} induces a quasi-isomorphism
\[
\pi_l\calF_{z}(U)\simeq \pi_l\Cou_{z}(U,L_\nis\calF)
\]
for all $z\in B$.
Our claim follows from 
the equivalences $\calF_{z}\simeq L_{\tf}(\calF_{z})\simeq (L_{\tf}\calF)_{z}$.
\end{proof}

\begin{corollary}\label{cor:CoutfsimeqCou}
For any $\calF\in \SH_{\A^1,\nis}(B)$ and essentially smooth local scheme $U\in \EssSm_B$, 
the canonical morphism \eqref{eq:CoutftoCou}, see \Cref{def:CoutftoCoubi_AND_CoutftoCou},
induces for all $l\in\bbZ$ a quasi-isomorphism
\[
\pi_l\Cou_\tf(U,\calF)\simeq \pi_l\Cou(U,\calF)
\]
in $\Kom(\Ab)$.
\end{corollary}
\begin{proof}
The claim follows from \Cref{th:CoutfsimeqCou} by \Cref{lm:SHAnistoSHfrsAnis}.
\end{proof}

\begin{theorem}\label{thm:trivcohhdimBCous}
For any
quasi-stable framed $\calF\in \SHsAtf(B)$ (resp. $\calF\in \SH_{\A^1,\nis}(B)$) and for any
essentially smooth local scheme $U\in \EssSm_B$ and any integer $l\in\bbZ$,
the cohomology of 
$\pi_l\Cou(U,L_\nis\calF)$ (resp. $\pi_l\Cou(U,\calF)$)
is trivial above the degree $\dim B$.
\end{theorem}

\begin{proof}
The claim follows from \Cref{th:CoutfsimeqCou} and \Cref{cor:CoutfsimeqCou} because the length of the 
complex of abelian groups $\pi_l\Cou_\tf(U,L_\tf\calF)$ is equal to $\dim B$.
\end{proof}

\begin{remark}
We note that $C^\bullet_\tf(-,\calF)$ is a complex of presheaves on $\Sm_B$.
\end{remark}

\begin{corollary}
There is a functor
\[\Sm_B\to D_\mathrm{Zar}(\Sm_B,\mathbb Z); X\mapsto C^\bullet(-,\calF)\] 
such that for any open immersion $j\colon U\to X$ the morphism $j^*$ agrees with 
the termwise morphisms of complexes $C^\bullet(X,\calF)\to C^\bullet(U,\calF)$ induced by morphisms
$\calF_x(X)\to \calF_x(U)$.
\end{corollary}

\section{Motivic localization and infinite loop spaces}
\label{sect:dimBgeneralisty}

As an application, we extend the generality of 
the main results in \cite{DKO:SHISpecZ} 
for 1-dimensional base schemes to arbitrary ones.
We begin with the fundamental
strict homotopy invariance theorem,
see \cite[Theorem 12.2, Theorem 16.6]{DKO:SHISpecZ}.
\begin{theorem}[$\tf/\nis$-strict homotopy invariance]\label{th:SHIsepnfdB}
If $\calF\in \SH^{\fr,s}_{\A^1,\tf}(B)$,
then $L_\nis\calF$ is $\A^1$-invariant.
The same holds for $\calF\in \SH^{\fr,s,t}_{\A^1,\tf}(B)$.
\end{theorem}
\begin{proof}
    As noted in \cite[Remark 1.22]{DKO:SHISpecZ} 
    the proof of \cite[Theorem 12.2]{DKO:SHISpecZ} 
    verifies our claim
provided for all $X\in\Sm_B$, $z\in B$, $\calF\in \SH^{\fr,s}_{\A^1}(z)$, 
    there is an equivalence
    \begin{equation}\label{eq:calFXhczsimeqLniscalFXhxz}\calF((X^h_x)_z)\xrightarrow{\calsimeq} L_\nis \calF((X^h_x)_z),\end{equation}
    or, equivalently,
    there are isomorphisms
    \begin{equation}\label{eq:HXhxzF}
    H^n_{\nis}((X^h_x)_z,\calF)\cong 
    \begin{cases} \calF_\nis((X^h_x)_z), & n=0\\
    0, & n>0.
    \end{cases}\end{equation}
    for all $X\in \Sm_B$, $z\in B$, and $\calF\in \Pre^{\Ab}_{\A^1}(\Corr^\fr(z))$.
    The proofs of \eqref{eq:calFXhczsimeqLniscalFXhxz} and \eqref{eq:HXhxzF}
    in \cite[\S 8]{DKO:SHISpecZ}
    use that $\dim B=1$.
    For a general 
    base scheme 
    we conclude using
    \Cref{lm:F(Uz)Fnis(Uz)}, see also \Cref{rem:lm:F(Uz)Fnis(Uz)}.
    
    Since each row of a bispectrum $\calF\in \SH^{\fr,s,t}_{\A^1,\tf}(B)$ defines an object of 
    $\SH^{\fr,s}_{\A^1,\tf}(B)$,
    the claim regarding $\SH^{\fr,s,t}_{\A^1,\tf}(B)$ follows as well.
\end{proof}

\begin{remark}
\Cref{lm:F(Uz)Fnis(Uz)} proves \eqref{eq:HXhxzF}, 
and holds for any essentially smooth local $B$-scheme
as mentioned in \Cref{rem:lm:F(Uz)Fnis(Uz)};
\eqref{eq:calFXhczsimeqLniscalFXhxz} and \eqref{eq:HXhxzF} hold for local schemes of the form $X_{(x)}$ instead of 
the local henselian scheme $X^h_x$ as required.
\end{remark}

\begin{theorem}\label{th:LnisOmegaGmLA1tfFrSGmXgp}
For any $X\in \Sm_B$, there is a natural equivalence
\[\Omega^\infty_{\PP^1}\Sigma^\infty_{\PP^1}X_{+}\simeq 
L_\nis \varinjlim_l \Omega^l_{\Gm}(L_{\A^1,\tf} \Fr(-, X_{+}\wedge\Gm^{\wedge l} ))^\gp.\]
\end{theorem}
\begin{proof}
 In view of the $\tf/\Nis$-strict homotopy invariance in \Cref{th:SHIsepnfdB}
 the claim follows by the argument for \cite[Theorem 16.6]{DKO:SHISpecZ}.
\end{proof}

\section{Zariski stalks of motivic homotopy groups}\label{sect:Zarstalks}

We will now improve the results of the computations of the motivic infinite loop spaces
and connectivity results from \cite{DKO:SHISpecZ},
obtaining results for motivic infinite loop spaces and stable motivic homotopy groups
on semi-local essentially smooth schemes.

To begin, we collect a few observations on semi-local schemes. 
For a finite set of points $P$ of a scheme $X$, denote by $X_{(P)}$ the semi-local scheme of $X$ concerning $P$.

\begin{lemma}\label{lm:upXP}
\par    (1)
    Let $X$ and $Y$ be schemes and $P$, $Q$ be finite sets of points of $X$ and $Y$, 
    respectively.
    If $f\colon X\to Y$ is a morphism that takes points in $P$ to points in $Q$, then $f$ induces a morphism $X_{(P)}\to Y_{(Q)}$.
\par    (2)
    Let $S$ be a semi-local scheme with set of closed points $P$.
    Then the canonical morphism $S_{(P)}\to S$ is an isomorphism.
\par    (3)
    Any morphism $f\colon S\to X$, for a semi-local scheme $S$,
    factors through the semi-local scheme $X_{f(P)}$,
    where $P$ is the set of closed points of $S$.
\end{lemma}
\begin{proof}
    Parts (1) and (2) are evident.
    Part (3) follows from (1) and (2).
\end{proof}

\begin{definition}\label{def:surjsmol}
A morphism of schemes $\widetilde X\to X$ is \emph{surjective on semi-local schemes} if
for any semi-local scheme $S$, the map
\[\mathrm{Map}(S,\widetilde X)\to \mathrm{Map}(S,X)\]
is surjective,
where
$\mathrm{Map}(-,-)$ denotes the set of morphisms of schemes.
Let $\mathrm{SSL}$ be
the class of morphisms of schemes
that is surjective on semi-local schemes.
\end{definition}

\begin{remark}\label{rem:SSL}
By \Cref{def:surjsmol}
the class $\mathrm{SSL}$ is equal to the class of morphisms that have the right lifting property with respect to 
the morphisms of the form $\emptyset\to S$ for all semi-local $S$.
\end{remark}

\begin{lemma}\label{lm:surjsloc}
A morphism of schemes $\widetilde X\to X$
belongs to $\mathrm{SSL}$ if and only if
for each finite set of points $P$ in $X$ there is a morphism $S\to\widetilde X$ such that the diagram \begin{equation}\label{eq:VPwidetildeVV}\xymatrix{& \widetilde X\ar[d]\\S\ar[r]\ar[ru] & X} \end{equation}
commutes,
where $S=X_{(P)}$.
\end{lemma}
\begin{proof}
    \Cref{def:surjsmol} claims in an equivalent form that for each semi-local scheme $S$
    there is a morphism $S\to\widetilde X$ such that the diagram \eqref{eq:VPwidetildeVV} commutes.
    So, one implication is immediate.
    The converse implication
    follows by \Cref{lm:upXP}(3).
\end{proof}

\begin{lemma}\label{lm:sLocliftpropertyfamilyclousnessprop}
    The following holds for $\mathrm{SSL}$.
    \begin{itemize}
        \item[(1)] The class $\mathrm{SSL}$ is closed under composition.
        \item[(2)] If $v_2\circ v_1\in \mathrm{SSL}$ for morphisms $V_0\xrightarrow{v_1}V_1\xrightarrow{v_2}V_2$, then $v_2\in \mathrm{SSL}$.
        \item[(3)] The class $\mathrm{SSL}$ is closed under pullbacks.
        \item[(4)] If $v^\prime,{v^\prime}^*(\tilde v)\in \mathrm{SSL}$ 
            for morphisms $v^\prime\colon V^\prime\to V$, $\tilde v\colon \widetilde V\to V$,
            then $\tilde v\in \mathrm{SSL}$.
    \end{itemize}
    Here we denote by ${v^\prime}^*(\tilde v)\colon V^\prime\times_V \widetilde V\to V^\prime$
    the pullback of $\tilde v$ along $v^\prime$.
\end{lemma}
\begin{proof}
    The claims (1)-(3) hold by \Cref{rem:SSL},
    or alternatively, 
    claims (1) and (2) follow from the similar properties for the class of surjective morphisms of sets,
    and claim (3) is straightforward.
    Next, we prove (4).
    Let $S$ be a semi-local scheme and $s\in \mathrm{Map}(S,V)$.
    Then there exist some $s^\prime\in \mathrm{Map}(S,V^\prime)$ such that $v^\prime s^\prime = s$, 
    and also $\tilde s^\prime\in \mathrm{Map}(S,V^\prime\times_V \widetilde V)$ such that 
    ${v^\prime}^*(\tilde v) \tilde s^\prime = s^\prime$.
    Hence, we have $\tilde v \tilde s = s$ for $\tilde s={\tilde v}^*(v^\prime)\tilde s^\prime\in \mathrm{Map}(S,\widetilde V)$.
    Thus, the map $\mathrm{Map}(S,\widetilde V)\to\mathrm{Map}(S,V)$ is surjective.
\end{proof}

\Cref{lm:sLocliftpropertyfamilyclousnessprop} implies that the class $\mathrm{SSL}$ 
defines a site on the category of schemes with enough points given by semi-local schemes.
We consider the intersection of this topology with the Zariski topology.

\begin{definition}\label{def:sZar}
Define the \emph{semi-Zariski topology} on $\Sch_B$ 
over a base scheme $B$
as the subtopology of the Zariski topology generated by the Zariski coverings $v\colon \widetilde X\to X$ in $\Sch_B$ that are surjective on semi-local schemes.
Denote the semi-Zariski topology by $\szar$.
\end{definition}
\begin{lemma}
A Zariski covering
$v\colon \widetilde X\to X$
is a semi-Zariski covering
if and only if
for each finite set of points $P$ in $X$, 
there exists a morphism $X_{(P)}\to\widetilde X$ such that the diagram \eqref{eq:VPwidetildeVV} 
commutes for $S=X_{(P)}$ --- the semi-local scheme of $X$ with respect to $P$.
The schemes $X_{(P)}$, for $X\in\Sm_B$, form a set of points for the $\szar$-topology on $\Sm_B$.
\end{lemma}
\begin{proof}
    The first claim follows by \Cref{lm:surjsloc}. To proceed with the second claim, 
    suppose that $f\colon\calF\to\calF^\prime$ is a morphism of presheaves of sets on $\Sm_B$ 
    inducing isomorphisms $\calF(X_{(P)})\cong\calF^\prime(X_{(P)})$ for all $X\in\Sm_B$ 
    and finite set of points $P\subset X$.
    Then, for $X$ and $P$ as above, 
    there exists a Zariski neighborhood $X^P$ of $P$ in $X$
    such that $f$ induces an isomorphism $\calF(X^P)\cong\calF^\prime(X^P)$.
    Since $\coprod_{P\subset X}X^P\to X$ is a $\szar$-covering by the above,
    $f$ induces an isomorphism $\calF(X)\cong\calF^\prime(X)$.
    Thus, $f$ is an isomorphism.
\end{proof}

We identify the $\infty$-category $\cPre(B)=\cPre(\Sm_B)$ 
with the subcategory of continuous presheaves in $\cPre(\EssSm_B)$.
Recall the functor $\gamma\colon\Sm_B\to \Corr^\fr(B)$
from \Cref{subsect:Corrtfr} and
denote by the same symbol the functor
$\gamma\colon\EssSm_B\to\Corr^\fr(\EssSm_B)$.
The following lemma justifies the notation $\gamma_*$ and $\gamma^*$
for the corresponding direct and inverse image functors.

\begin{lemma}\label{lm:gammalsuscont}
The functors 
$\gamma_*\colon \cPrefr(\EssSm_B)\to\cPre(\EssSm_B)$ and $\gamma^*\colon \cPre(\EssSm_B)\to\cPrefr(\EssSm_B)$ 
preserve continuous presheaves.
\end{lemma}

\begin{proof}
    By the definition
    a presheaf $\calF$ in $\cPrefr(\EssSm_B)$ is continuous if and only if
    $\gamma_*(\calF)$ is continuous.
    To prove the claim on $\gamma^*$ we note that $h^\fr(X)\in\cPrefr(\EssSm_B)$ 
    is continuous for all $X\in\Sm_B$: 
    It suffices to prove the continuity of $h^\fr(X)$ on affine $B$-schemes.
    Indeed, 
    as shown in \cite[Proposition 2.3.23]{five-authors} 
    each span (as in \Cref{def:tgfr}) 
    \[X\leftarrow Z\rightarrow Y\] 
    for affine $X$
    equals the span
    for some normally framed correspondence 
    \cite[Definition 2.2.2]{five-authors}.
    Since 
    the presheaves of normally framed correspondences $h^\mathrm{nfr}(-,X)$ are continuous on affine $B$-schemes,  
    see \cite[Theorem 5.1.5]{hty-inv},
    and the $K$-theory space presheaf $K(-)$ is continuous on the category of affine schemes,
    the claim follows.
    The representable presheaves generate the category $\cPre(\Sm_B)$ via colimits.
    Hence, the subcategory of continuous presheaves in $\cPre(\EssSm_B)$ is the image of the inverse image functor $\cPre(\Sm_B)\to\cPre(\EssSm_B)$, while the latter functor preserves colimits.
    The subcategories of continuous presheaves in $\cPre(\EssSm_B)$ and $\cPrefr(\EssSm_B)$ are closed with respect to colimits and are generated by presheaves $h(X)\in\cPre(\EssSm_B)$ and $h^\fr(X)\in\cPrefr(\EssSm_B)$ for all $X\in\Sm_B$.
    This finishes the proof.
\end{proof}

By definition,
a presheaf $\calF\in\cPre^{\fr}(B)$ is a \emph{semi-Zariski sheaf} 
if $\gamma_*\calF$ is a semi-Zariski sheaf, 
while the class of semi-Zariski local isomorphisms in $\cPre^{\fr}(B)$  
is generated by the image of the class of semi-Zariski isomorphisms under $\gamma^*$.

\begin{proposition}\label{prop:gammaLZarfrLZargammaHHfrAtf}
For any $\calF\in\cPre^{\fr}(B)$, there is an isomorphism
\[\gamma_* L_\szar^\fr \calF\simeq L_{\szar}\gamma_* \calF.\]   
The same holds for $\calF\in\cSpt^{\fr}(B)$.
\end{proposition}

\begin{proof}
    Consider the commutative diagram of $\infty$-categories
    \[\xymatrix{\Corr^\fr(\EssSm^{sLoc}_B)\ar[r]& \Corr^\fr(\EssSm_B)\\ 
    \EssSm^{sLoc}_B\ar[u]^-{\gamma_{sloc}}\ar[r] & \EssSm_B\ar[u]_-\gamma.}\]
    Let
    $\alpha$ be a semi-Zariski local isomorphism in $\cPre(B)$.
    Since $\gamma^*_{sloc}$ preserves isomorphisms, it follows that  
    $\gamma^*(\alpha)$ is an isomorphism on semi-local schemes.
    So
    $\gamma_*\gamma^*(\alpha)$
    is a semi-Zariski local isomorphism. 
    Thus 
    since the class of semi-Zariski local isomorphisms in $\cPre^{\fr}(B)$ is generated by 
    morphisms of the from $\gamma^*(\alpha)$ as above,
    it follows that $\gamma_*$ preserves semi-Zariski isomorphisms.
    Then
    since $\gamma_*$ preserves semi-Zariski sheaves, 
    it follows that $\gamma_*$ commutes with $L_{\szar}$.
    
    The claim for $\cSpt^{\fr}(B)$
    follows from the one for $\cPre^{\fr}(B)$
    because 
    $L_\szar^\fr$ and $\gamma^\fr_*$ commute with 
    $\Omega\colon\cSpt^{\fr}(B)\to\cSpt^{\fr}(B)$
    and
    $\Omega^\infty\colon\cSpt^{\fr}(B)\to\cPre^{\fr}(B)$,
    and similarly
    for $L_\szar$ and $\gamma_*$.
\end{proof}

The following result 
holds for an infinite field $k$
and a local scheme $U=X_{(x)}$
according to \cite[Theorem 3.15(3)]{hty-inv}, 
while 
for semi-local schemes 
the claim is proven in \cite[Theorem 7.0.1]{ccorrs} 
for presheaves with cohomological correspondence transfers.
\begin{theorem}\label{th:slinj}
For any semi-local essentially smooth scheme $U$ over a field $k$
and quasi-stable framed radditive abelian presheaf $\calF$ on $\Sm_k$, 
the homomorphism $\calF(U)\to\calF(U^{(0)})$ is injective.
\end{theorem}

\begin{proof}[Proof]
It suffices to consider semi-local schemes of the form $U=X_{(x_0,\dots,x_l)}$ for $X\in\Sm_k$, 
$x_0,\dots,x_l\in X$, 
by taking filtered limits.
By the argument of \cite[Theorem 3.15(3)]{hty-inv} it suffices to construct, 
for any closed subscheme $Y\not\hookrightarrow X$ of positive codimension,
and for $N\gg 0$, 
a framed correspondence \[c\in\ZF_N(U, X-Y)\] such that the $\A^1$-homotopy class in $\ZF(U,X)$ 
of the composite of $c$ with the open immersion $j\colon X-Y\to X$ equals the class of the $\sigma$-suspension 
of the canonical morphism $\mathrm{can}\colon U\to X$, i.e.,
\begin{equation}\label{eq:cj}[j\circ c]=[\mathrm{can}]\in \overline\ZF(U,X)\end{equation}
where $\ZF(U,X)=\varinjlim_{n}\ZF_n(U,X)$ and $\overline\ZF(U,X)=\ZF(U,X)/\ZF(U\times\A^1,X)$.
Suppose that $k$ is infinite. Let us start with the diagram 
\[
X\xleftarrow{v}\calC\xrightarrow{j}\overline\calC\xrightarrow{p}
U\]
and an ample line bundle $\calO_{\overline\calC}(1)$ on $\overline\calC$
provided by \cite[Lemma 7.0.5]{ccorrs} that satisfies the respective list of properties (1)-(6).
So, 
according to the properties (1),(4) $p$ is a projective morphism of dimension one that is smooth over $\calC$, 
and $D=\overline\calC\setminus\calC$ is finite over $U$.
The argument from \cite[Proof of Lemma 7.0.3]{ccorrs}
gives us sections
\begin{center}
\begin{tabular}{@{}>{$}l<{$}>{$}l<{$}>{$}l<{$}>{$}l<{$}@{}}\toprule
s\in \Gamma(\overline\calC,\calO(n) ) &
\widetilde s\in \Gamma(\overline\calC\times\stackrel{\lambda}{\A^1},\calO(l) ) &
s^\prime \in \Gamma(\overline\calC,\calO(l)\otimes\mathscr L(\Delta)^{-1}) &
\delta\in\Gamma(\overline\calC,\mathscr L(\Delta))  \\
\midrule
Z(s|_{\calZ\amalg D})=\varnothing & \widetilde s|_{\overline\calC\times 0}=s & Z(s'|_{\calZ\amalg D\amalg\Delta})=\varnothing & Z(\delta)=\Delta \\
 & \widetilde s|_{\overline\calC\times1}=s'\otimes\delta &  & \\
& \widetilde s|_{D\times\A^1}=s & &\\
\bottomrule
\end{tabular}
\end{center}
Here $\calZ=v^{-1}(Y)$, $\Delta$ is the graph of $\mathrm{can}$, and $\mathscr L(\Delta)$ is the line bundle on $\overline \calC$ corresponding to the divisor $\Delta$.
Consider a closed immersion $\overline\calC\not\hookrightarrow\PP^N_U$ provided by the ampleness of 
$\mathcal O_{\overline\calC}(1)$,
and a retraction $r\colon V\to\calC$ for an \'etale neighborhood $V$ of ${\calC}$ in $\A^N_U$ provided by smoothness of $\calC$, see \cite[Theorem I.8]{Gru72}.
Let $\phi\in \mathcal O_{V}(V)$, $\widetilde\phi \in \mathcal O_{V\times\A^1}(V\times\A^1)$ 
be such that \[\phi\big|_{\calC}=s/d^l, \quad \widetilde\phi\big|_{\calC}=\widetilde s/d^l,\]
and define $\phi^\prime$ as the inverse image of $\widetilde\phi$ along the composite 
$\A^N_U\cong\A^N_U\times\{1\}\hookrightarrow\A^N_U\times\A^1$.
Denote by $\mathrm{pr}_V\colon {V\times\A^1}\to{V}$ and $\mathrm{pr}_U\colon {U\times\A^1}\to{U}$ the canonical projections.
We write $e$ for $v\circ r$ and its restrictions on open subschemes of $V$, and $\widetilde e$ for $v \circ  r\circ\mathrm{pr}_V$ and its restrictions on open subschemes of $V\times\A^1$.
We define the framed correspondences
\begin{align*}
\Upsilon& = (Z(s'), V-\Delta, \phi^\prime ,e)\in\ZF_{N}(U,X-Y),\\ 
\Phi'&= (Z(s), V, \phi ,e)-\Upsilon\in\ZF_{N}(U,X-Y),\\
\Theta'&= (Z(\widetilde s), \widetilde\phi ,\widetilde e)-\Upsilon\circ \mathrm{pr}_U\in\ZF_N(U\times\A^1,X).
\end{align*}
Similarly to \cite[Proof of Lemma  7.0.3]{ccorrs}, the properties of the sections above imply that
\[
\Theta^\prime\circ i_0=j\circ \Phi^\prime
\in \ZF(U,X)
,
\quad
[\Theta^\prime\circ i_1]=
[\langle\nu\rangle^*(\mathrm{can})]\in \overline\ZF(U,X)
\]
and
for some $\nu\in k[U]^\times$ (see \Cref{ex:langleurangle} for $\langle\nu\rangle$).
Letting
$\Phi=\langle\nu^{-1}\rangle^*(\Phi^\prime)$, 
$\Theta= \langle\nu^{-1}\rangle^*(\Theta')$,
we obtain \eqref{eq:cj} 
since
\[
\Theta\circ i_0=j\circ \Phi,\quad
[j]=[\Theta\circ i_1].
\]

For a finite field $k$, the claim follows from the result over infinite extensions of $k$ by the argument in
\cite[Proof of Lemma 6.19]{nonperfect-SHI}.
\end{proof}

\begin{theorem}
\label{th:Prefr(k)SGA1nis_sZar}

There are equivalences
of endofunctors
\begin{equation}\label{eq:gammaOmegaSigmazar:k}
\OmegaSigma_{\A^{1},\nis} 
\calsimeq  
L_{\nis} L_{\A^1}\gamma_*\gamma^*
\calsimeq
L_{\szar} L_{\A^1} \gamma_*\gamma^*
\end{equation}
on the $\infty$-categories $\cPre(k)$ and $\cSpts(k)$, 
and
equivalences
of endofunctors
\begin{equation}\label{eq:OmegaSigmazf:k}
\OmegaSigma^{\fr}_{\A^{1},\nis} 
\calsimeq
L_{\nis}^\fr L_{\A^1}
\calsimeq
L_{\szar}^\fr L_{\A^1} 
\end{equation}
on the $\infty$-categories $\cPrefr(k)$ and $\cSptsfr(k)$.

\end{theorem}
\begin{proof}
    The first equivalences in \eqref{eq:gammaOmegaSigmazar:k} and \eqref{eq:OmegaSigmazf:k} are provided by \cite[Thereoms 3.5.14, 3.5.16]{five-authors}.
    Hence
    \[\OmegaSigma^{\fr}_{\A^{1},\nis}\calF(E)\calsimeq L_{\A^1}\calF(E)\] 
    for any $\calF\in\cPrefr(k)$ and $0$-dimensional reduced schemes $E\in\EssSm_k$
    because the Nisnevich topology on the small \'etale site over $E$ is trivial.
    Since the functors $\OmegaSigma^{\fr}_{\A^{1},\nis} $ and $L_{\A^1}$ land in $\A^1$-invariant objects, 
    \Cref{th:slinj} implies there is an equivalence
    \[\OmegaSigma^{\fr}_{\A^{1},\nis}\calF(U)\calsimeq L_{\A^1}\calF(U)\]
    for any semi-local $U\in\EssSm_k$.
    This implies the second isomorphism in \eqref{eq:gammaOmegaSigmazar:k}.
    The second isomorphism in \eqref{eq:OmegaSigmazf:k} follows by \Cref{prop:gammaLZarfrLZargammaHHfrAtf}.
\end{proof}

The $\zf$-topology is generated by the Zariski topology and the $\tf$-topology introduced in \cite{SHfrzf}.
We also consider the $\szf$-topology generated by the semi-Zariski topology and the $\tf$-topology.
That is, we have 
\[
\zf=\zar\cup\tf,\quad \szf=\szar\cup\tf.
\]
Next, we recall the reconstruction theorem \cite[Theorem 3.5.12]{five-authors} and extend to the $\szf$-topology 
an equivalence shown for the $\zf$-topology in \cite{SHfrzf}.

\begin{theorem}[\protect{\cite[Theorem 3.5.12]{five-authors}, \cite[Theorem 4.42]{SHfrzf}}]\label{th:SHfrszar}
There are equivalences
\[
\SH_{\A^1,\nis}(B)\simeq
\SH^\fr_{\A^1,\nis}(B)\simeq
\SH^\fr_{\A^1,\zf}(B)\simeq
\SH^\fr_{\A^1,\szf}(B)
\]
\end{theorem}
\begin{proof} 
The first equivalence is \cite[Theorem 3.5.12]{five-authors}, and the second one is \cite[Theorem 4.42]{SHfrzf}. 
To prove the third equivalence, we use induction on the Krull dimension of $B$. 
Applying the localization theorem \cite[Theorem 3.18]{SHfrzf} to the $\szf$-topology, 
we may assume that $B$ is a field and conclude by appealing to \Cref{th:Prefr(k)SGA1nis_sZar}.
\end{proof}

Consider the subcategories of $\tf$-local objects $\cPrefrtf(B)$ and $\cSptsfrtf(B)$ in
the $\infty$-categories of presheaves of spaces $\cPrefr(B)$ and of spectra $\cSptfr(B)$ 
on the $\infty$-category $\Corr^\fr(B)$.
Consider 
the
$\infty$-category of respective bispectra
$\cSpt^{\fr,s,t}_\tf(B)=\cSptsfrtf(B)[\Gm^{\wedge 1}]$,
and the adjunctions of $\infty$-categories
\begin{equation}\label{eq:SigmaOmega_tf_GmANis}
\cSptsfrtf(B) 
\rightleftarrows
\cSptstfrtf(B)
\rightleftarrows
\cSptstfrAN(B).
\end{equation}
We denote by $\OmegaSigma^{\tf,\fr}_{\A^{1},\nis}$
the unit of the composite adjunction of \eqref{eq:SigmaOmega_tf_GmANis}
and use the same notation for the adjunction
\[
\cPrefrtf(B) 
\rightleftarrows
\cSptsfrtf(B).
\]

\begin{proposition}
\label{prop:Pretffr(B)SGA1nis_Zar}
There is a natural isomorphism of endofunctors
\begin{equation}\label{eq:OmegaSigmazf}
\OmegaSigma^{\tf,\fr}_{\A^{1},\nis} 
\calsimeq 
L_{\szf}^\fr\OmegaSigma^{\tf,\fr}_{\A^{1},\tf} 
\end{equation}
on the $\infty$-categories $\cPrefrtf(B)$ and $\cSptsfrtf(B)$.

For any 
$\calF\in\cPrefrtf(B)$ (resp. $\cSptsfrtf(B)$)
there is an equivalence
\begin{equation}\label{eq:gammaOmegaSigmazar}
\gamma_*\OmegaSigma^{\tf,\fr}_{\A^{1},\nis}\calF 
\calsimeq
L_{\szar}\gamma_*\OmegaSigma^{\tf,\fr}_{\A^{1},\tf}\calF 
\end{equation}
in $\cPretf(B)$ (resp. $\cSptstf(B)$). 
\end{proposition}

\begin{proof}
Consider the case of \eqref{eq:OmegaSigmazf} for the $\infty$-category $\cPrefrtf(B)$.
By \cite[Theorem 15.10]{DKO:SHISpecZ}  we have
\[
\OmegaSigma^{\tf,\fr}_{\A^{1},\nis} 
\calsimeq 
L_{\nis}^\fr\OmegaSigma^{\tf,\fr}_{\A^{1},\tf}. 
\]
Then, by \cite[Theorem 15.4]{DKO:SHISpecZ},
$\OmegaSigma^{\tf,\fr}_{\A^{1},\tf}$ preserves Nisnevich local equivalences.
Therefore, $\OmegaSigma^{\tf,\fr}_{\A^{1},\tf}$ preserves Zariski local equivalences,
and hence also $\szf$-local equivalences.
Using \cite[Theorem 15.4]{DKO:SHISpecZ} 
we conclude that
\[
\OmegaSigma^{\tf,\fr}_{\A^{1},\szf} 
\simeq
L_{\szf}^\fr\OmegaSigma^{\tf,\fr}_{\A^{1},\tf} 
\]
Finally, since $\SH^\fr_{\nis}(B)\simeq \SH^\fr_{\szf}(B)$ by \Cref{th:SHfrszar},
there is an equivalence
\[\OmegaSigma^{\tf,\fr}_{\A^{1},\nis}\simeq
\OmegaSigma^{\tf,\fr}_{\A^{1},\szf}.\]
This implies \eqref{eq:OmegaSigmazf} for $\cPrefrtf(B)$.
We use this to prove our claim for $\cSptsfrtf(B)$:
Since the $\infty$-category $\Corr^\fr(B)$ is semi-additive, see \cite[Lemma 3.2.5]{five-authors},
$\cSptsfr(B)$ is equivalent to the subcategory $\cPrefr(B)^\mathrm{gp}$ of group-like objects in $\cPrefr(B)$,
and, consequently,
$\cSptsfrtf(B)\simeq\cPrefrtf(B)^\mathrm{gp}$.

The equivalence \eqref{eq:gammaOmegaSigmazar} follows from \eqref{eq:OmegaSigmazf}
by \Cref{prop:gammaLZarfrLZargammaHHfrAtf}.
\end{proof}

Denote by $\OmegaSigma^{}_{\A^{1},\nis}$
the unit of the composite adjunction of the sequence
\[
\cPre(B) 
\rightleftarrows
\cSpts(B)
\rightleftarrows
\cSptst(B)
\rightleftarrows
\cSptstAN(B).
\]
For a scheme $X\in\Sm_B$, 
we denote by the same symbol
the representable presheaf in $\cPre(B)$.
The standard loop-suspension adjoint functors $\Sigma^\infty_{\PP^1}$ and $\Omega^\infty_{\PP^1}$ yield the equivalence
\[
\OmegaSigma^{}_{\A^{1},\nis}X\calsimeq\Omega^\infty_{\PP^1}\Sigma^\infty_{\PP^1}X_+.\]
\begin{theorem}\label{th:LzarOmegaGmLA1tfFrSGmXgp}
For any $X\in \Sm_B$, there are natural equivalences in $\catPre(B)$
\begin{equation}\label{eq:ThetaLzarLsZar}
\begin{array}{lcl}
\OmegaSigma^{}_{\A^{1},\nis}X&\simeq& 
L_\szar \varinjlim_l \Omega^l_{\Gm}(L_{\A^1,\tf} \Fr(-, X_{+}\wedge\Gm^{\wedge l} ))^\gp\\
&\simeq&L_\zar \varinjlim_l \Omega^l_{\Gm}(L_{\A^1,\tf} \Fr(-, X_{+}\wedge\Gm^{\wedge l} ))^\gp.
\end{array}\end{equation}
For any semi-local scheme $U\in\EssSm_B$, there is an equivalence of spaces
\begin{equation}\label{eq:ThetasemilclimGmLatfFr}
(\OmegaSigma^{}_{\A^{1},\nis}X)(U)\simeq
\varinjlim_l \Omega^l_{\Gm}(L_{\A^1,\tf} \Fr(U, X_{+}\wedge\Gm^{\wedge l} ))^\gp.\end{equation} 
\end{theorem}

\begin{proof}
\Cref{prop:Pretffr(B)SGA1nis_Zar} shows the first equivalence in \eqref{eq:ThetaLzarLsZar}, 
cf. \cite[Theorem 16.6]{DKO:SHISpecZ}), 
which in turn implies the equivalence in \eqref{eq:ThetasemilclimGmLatfFr}.
Moreover, 
the second equivalence in \eqref{eq:ThetaLzarLsZar}
follows from \eqref{eq:ThetasemilclimGmLatfFr} for local schemes $U\in\EssSm_B$.
\end{proof}

Recall that $\pi^{\A^1,\nis}_{i,j}(\Sigma^\infty_{\PP^1} X_+)$
denotes the presheaf of stable motivic homotopy groups on $\Sm_B$ for $X\in\Sm_B$,
see \eqref{eq:piijANisCalFU}.
\begin{theorem}\label{cor:ShifteddimBConnectivityZar}
    Suppose that $X\in\Sm_B$, 
    $i<-\dim B+j$, and $j\in\bbZ$.
    Then, for any essentially smooth semi-local
     $B$-scheme $U$,
    we have
    \[\pi^{\A^1,\nis}_{i,j}(\Sigma^\infty_{\PP^1} X_+)(U)=0.\]
    Equivalently, 
    the presheaf
    $\pi^{\A^1,\nis}_{i,j}(\Sigma^\infty_{\PP^1} X_+)$
    is semi-Zariski locally trivial.
\end{theorem}

\begin{proof}
    Since by \cite[Proposition 16.10]{DKO:SHISpecZ} 
    the functor $L_{\A^1,\tf}$ takes connective objects to $(-\dim B)$-connective ones,
    the claim follows from \Cref{th:LzarOmegaGmLA1tfFrSGmXgp}.
\end{proof}

\section{A family of non-acyclic Cousin complexes}\label{section:counterex}
In this section, we aim to show that the Cousin complex over a positive dimensional base is generally non-acyclic. 
Throughout,   
the base scheme $B$ has dimension $d>0$,
and except for in \Cref{th:nontrivtermsBBz}
we assume that $B$ is local with closed point $z$.
We will work with 
\begin{equation}\label{eq:calFOmegaSimeqBBz}
\calF = \Omega^\infty_{\Gm}\SigmainftyT(B/(B-z))\in\SHsAN(B).
\end{equation}
We aim to show that the associated Cousin complex $\pi_l C^\bullet(B,\calF)$ is non-acyclic for $l=\codim_B z$.

Given $U\in\EssSm_B$,
a non-empty closed immersion $Z\not\hookrightarrow B$, 
with $U_Z:=U\times_B Z$,
we consider the canonical map of motivic spaces
\[U/(U-U_Z)\to B/(B-Z)\to B/(B-z).\]
We denote by $c_Z$ the corresponding element in 
\begin{equation*}
\label{eq:nontrivgroup}
[\SigmainftyT U/(U-U_Z),\SigmainftyT(B/(B-z))]_{\SHPAN(B)}
\cong\pi_0\calF_{U_Z}(U) =\pi_0\calF_{Z}(U).
\end{equation*}

The isomorphism holds by \eqref{eq:calFOmegaSimeqBBz}, 
and the equality holds because $\calF_Z=\calF_{U_Z}$ by \Cref{def:calFZU}.

\begin{lemma}\label{lm:nontrivclass}
For any $U\in\EssSm_B$ such that $U\times_B z\neq\emptyset$,
and non-empty closed subscheme $Z$ in $B$,
the element $c_Z\in \pi_0\calF_{Z}(U)$ is non-trivial.
\end{lemma}
\begin{proof}
Base change along the closed embedding $z\not\hookrightarrow B$
takes $c_Z$ to the class in \[[\Sigma^\infty_{\PP^1}(U_z)_+, \Sigma^\infty_{\PP^1}z_+]_{\SHsAN(z)}\] defined by the canonical morphism of schemes $U_z\to z$, 
which is non-trivial because $U_z\neq\emptyset$ by assumption.
\end{proof}

\begin{lemma}\label{lm:trivofclasses}
For any $U\in\EssSm_B$
and
non-empty closed subscheme $Z$ in $B$,
there are stable equivalences of spectra
\[\calF_{Z-z}(U-U_z)\simeq 0,\quad \calF_{z}(U)\simeq \calF_{Z}(U).\]
\end{lemma}

\begin{proof}
The first equivalence follows because
\[\varinjlim_{\codim_B Z=d-1}\calF_{Z-z}(U-U_z)\calsimeq
\varinjlim_{\codim_B Z=d-1}(u^*\calF)_{Z-z}(U-U_z)\calsimeq 0,\]
where $u\colon B-z\to B$,
because
$u^*\calF \simeq 0\in \SHs(B-z)$.
The second equivalence follows from the first
because 
\[\calF_{z}(U)\calsimeq\fib(\calF_{Z}(U)\to\calF_{Z-z}(U-U_z)).\]
\end{proof}

\begin{proposition}
\label{prop:differentialtoclp_ld}
For any $U\in\EssSm_B$ such that $U\times_B z\neq\emptyset$,
consider the morphisms
\begin{equation}\label{eq:differentialtoclp_ld}
\partial\colon\bigoplus\limits_{\codim_B x=d-1}\pi_{l-(d-1)}\calF_{x}(U)\to \bigoplus\limits_{\codim_B x=d}\pi_{l-d}\calF_{x}(U)
\end{equation}
defined in \eqref{eq:partialzprimez}.
Then, for $l=d$, there is an isomorphism
\begin{equation}
\label{eq:cokerpartialdFzU}
\operatorname{Coker}(\partial)\cong\pi_{0}\calF_{z}(U).
\end{equation}
The element $c_z\in \pi_{0}\calF_{z}(U)$ is non-trivial.
\end{proposition}
\begin{proof}
Since $B$ is local, we have the isomorphisms 
\begin{align*}
\bigoplus\limits_{\codim_B x=d-1}\pi_{l-(d-1)}\calF_{x}(U)&\cong\varinjlim_{\codim_B Z=d-1}\pi_{l-(d-1)}\calF_{Z-z}(U-U_z),\\
\bigoplus\limits_{\codim_B x=d}\pi_{l-d}\calF_{x}(U)&\cong\pi_{l-d}\calF_{z}(U).
\end{align*}
Thus, the left-hand side of \eqref{eq:differentialtoclp_ld} is trivial 
by applying \Cref{lm:trivofclasses} to $Z=z$,
and \eqref{eq:cokerpartialdFzU} follows.
The second claim follows by applying \Cref{lm:nontrivclass} to $Z=z$.
\end{proof}

\begin{proposition}\label{prop:differentialtoclp_lmo}
For any $U\in\EssSm_B$ such that $U\times_B z\neq\emptyset$,
the morphism
\begin{equation}
\label{eq:differentialtoclp_lmo}
g\colon\pi_{l}\calF(U)\to \bigoplus\limits_{\codim_B x=0}\pi_{l}\calF_{x}(U)
\end{equation}
induces, for $l=0$, an isomorphism
\begin{equation}\label{eq:kerpartial0}
\operatorname{Ker}(g)\cong \pi_{0}\calF(U).   
\end{equation}
The element $c_B\in \pi_{0}\calF(U)$ is non-trivial 
(under the canonical equivalence $\calF(U)\simeq \calF_B(U)$).
\end{proposition}
\begin{proof}
For $l=0$ the left-hand side of \eqref{eq:differentialtoclp_lmo} is $\pi_0\calF(U)$.
The element $c_B\in\pi_0\calF(U)$ is non-trivial
by \Cref{lm:nontrivclass} applied to $Z=B$.
The right-hand side of \eqref{eq:differentialtoclp_lmo} is trivial
by \Cref{lm:trivofclasses}
applied to all closed subschemes $Z$ in $B$ of codimension $1$.
\end{proof}

For any $U\in\EssSm_B$,
we consider 
the $d+2$-term complex of abelian groups 
\begin{equation*}\label{eq:Cou(B,F)}\pi_l \ovCou_\tf(U,\calF)\end{equation*}
concentrated in degrees $-1,\dots,d$,
see \Cref{def:tfCou,def:ovCoutf}.

\begin{theorem}
\label{cor:nontrivtermsBBz}
Suppose $B$ is a local scheme of Krull dimension $d>0$ with closed point $z$.
For any non-empty $U\in\EssSm_B$,
the cohomology group
\[
H^{i}(\pi_l \ovCou_\tf(U,\calF)) \cong \begin{cases}
\pi_{0}\calF(U),  & \text{if }i=-1\text{ and } l=0,\\
0,  & \text{if }i=0,\dots,d-1\text{ and } l\in\bbZ,\\
\pi_{0}\calF_{z}(U),  & \text{if }i=d\text{ and } l=d.
\end{cases}
\]
The element $c_B\in\pi_{0}\calF(U)$ defined by the canonical equivalence $\calF(U)\simeq \calF_B(U)$
is non-trivial.
Moreover, 
$c_z\in\pi_{0}\calF_{z}(U)$ is non-trivial.
\end{theorem}

\begin{proof}
The claim
for
$i=0,\dots,d-1$ and $l\in\bbZ$
follows from \Cref{lm:trivofclasses}.
The other claims
follow from \Cref{prop:differentialtoclp_lmo,prop:differentialtoclp_ld}.
\end{proof}

\begin{theorem}\label{th:nontrivtermsBBz}
Suppose $B$ is any base scheme of positive Krull dimension.
Let $U\in\EssSm_B$ be an essentially smooth local $B$-scheme. 
Assume the image $z\in B$ of the closed point of $U$ has positive codimension in $B$.
Then
\[
H^{i}(\pi_l\ovCou( U, \calF ))\neq 0,\quad i=\codim_B z,
\]
where $l=\codim_B z$, $\calF=\Omega^\infty_{\Gm}\SigmainftyT(B/(B-Z))$, and $Z$ is the closure of $z$ in $B$.
\end{theorem}
\begin{proof}
This follows by applying \Cref{cor:nontrivtermsBBz} to the local scheme $B_{(z)}$ of $B$ at $z$. 
\end{proof}

\printbibliography
\end{document}